\documentclass{amsproc}

\usepackage{latexsym, amssymb, amscd, rotating}

\newtheorem{theorem}{Theorem}[section]
\newtheorem{lemma}[theorem]{Lemma}
\newtheorem{prop}[theorem]{Proposition}
\newtheorem{cor}[theorem]{Corollary}

\theoremstyle{definition}

\theoremstyle{remark}
\newtheorem{remark}[theorem]{Remark}

\def \para#1{\par\medskip\textbf{#1}
              \addtocounter{theorem}{1}}

\numberwithin{equation}{section}

\renewcommand\a{\alpha}
\renewcommand\b{\beta}
\newcommand\g{\gamma}
\renewcommand\d{\delta}
\newcommand\la{\lambda}
\newcommand\z{\zeta}
\newcommand\e{\eta}
\renewcommand\th{\theta}
\newcommand\io{\iota}

\newcommand\s{\sigma}
\newcommand\x{\chi}
\newcommand\f{\phi}
\newcommand\vf{\varphi}
\newcommand\p{\psi}
\renewcommand\t{\tau}
\renewcommand\r{\rho}

\newcommand\vS{\varSigma}
\newcommand\D{\Delta}
\newcommand\vD{\varDelta}

\newcommand\F{\Phi}

\newcommand\vL{\varLambda}

\newcommand\ve{\varepsilon}

\newcommand{\FF}{\mathbb F}

\newcommand{\ZZ}{\mathbb Z}

\newcommand\Fq{{\mathbf F}_q}

\newcommand\Ql{\bar{\mathbf Q}_l}
\newcommand\BQ{\mathbf Q}

\newcommand\BF{\mathbf F}

\newcommand\BZ{\mathbf Z}

\newcommand\Br{\mathbf r}

\newcommand\ZC{\mathcal{C}}
\newcommand\CH{\mathcal{H}}
\newcommand\CI{\mathcal{I}}
\newcommand\CJ{\mathcal{J}}
\newcommand\CE{\mathcal{E}}
\newcommand\CL{\mathcal{L}}
\newcommand\DD{\mathcal{D}}
\newcommand\CS{\mathcal{S}}

\newcommand\CM{\mathcal{M}}

\newcommand\CP{\mathcal{P}}
\newcommand\CQ{\mathcal{Q}}

\newcommand\CF{\mathcal{F}}

\newcommand\CZ{ \mathcal{Z}}

\newcommand\CW{ \mathcal{W}}

\newcommand\iv{^{-1}}

\newcommand\wh{\widehat}
\newcommand\wt{\widetilde}
\newcommand\wg{^{\wedge}}

\newcommand\ol{\overline}

\newcommand\lra{\leftrightarrow}

\newcommand\ssim{/\!\!\sim}

\newcommand\IC{\operatorname{IC}}

\newcommand\End{\operatorname{End}}
\newcommand\ind{\operatorname{ind}}
\newcommand\Ind{\operatorname{Ind}}

\newcommand\Irr{\operatorname{Irr}}
\newcommand\Sh{\operatorname{Sh}}

\newcommand\Tr{\operatorname{Tr}\,}

\newcommand\ad{\operatorname{ad}}
\newcommand\ex{_{\operatorname{ex}}}

\newcommand\der{_{\operatorname{der}}}

\newcommand\uni{_{\operatorname{uni}}}

\newcommand\lp{\operatorname{\!\langle\!}}
\newcommand\rp{\operatorname{\!\rangle}}

\newcommand\per{^{\bot}}
\newcommand\ev{\operatorname{ev}}
\newcommand\odd{\operatorname{odd}}

\newcommand\dw{\dot w}
\newcommand\da{\dot a}

\newcommand{\isom}{\,\raise2pt\hbox{$\underrightarrow{\sim}$}\,}
\begin{document}

\title{Lusztig's conjecture for finite classical groups \\
       with even characteristic}

\author{Toshiaki Shoji}
\address{Graduate School of Mathematics, Nagoya University, 
         Chikusa-ku, Nagoya 464-8602, Japan}





\begin{abstract}
The determination of scalars involved in 
Lusztig's conjecture concerning the characters 
of finite reductive groups was achieved by Waldspurger 
in the case of finite classical groups $Sp_{2n}(\Fq)$ or
$O_n(\Fq)$ when $p,q$ are large enough.  Here $p$ is the 
characteristic of the finite field $\Fq$.
In this paper, we determine the scalars in the case 
of $Sp_{2n}(\Fq)$ with $p = 2$, by applying the theory of 
symmetric spaces over a finite field due to 
Kawanaka and Lusztig.  We also obtain a weaker result for 
$SO_{2n}(\Fq)$ with $p = 2$, of split type.
\end{abstract}

\maketitle

\renewcommand{\theenumi}{\roman{enumi}}
\renewcommand{\labelenumi}{(\theenumi)}
\renewcommand{\thefootnote}{\fnsymbol{footnote}}

\par\bigskip
\addtocounter{section}{-1}
\section{Introduction}
Let $G$ be a connected reductive group defined over 
a finite field $\Fq$ of characteristic $p$ 
with Frobenius map $F$.  Lusztig's 
conjecture asserts that, under a suitable parametrization, 
almost characters of the finite reductive group 
$G^F$ coincide with the characteristic 
functions of character sheaves of $G$ up to scalar. 
Once Lusztig's conjecture is settled, and the scalars involved
there are determined, one obtains a uniform algorithm 
of computing irreducible characters of $G^F$.
Lusztig's conjecture was solved in [S1] in the case where
the center of $G$ is connected. In [S2], the scalars 
in question were determined in the case where $G$ is 
a classical group with connected center, when $p$ is odd, 
and the scalars are related to the unipotent characters of $G^F$.
By extending the method there, Waldspurger [W] proved 
Lusztig's conjecture (or its appropriate generalization) 
for $Sp_{2n}$ and $O_n$ assuming that $p, q$ are large enough.
He also determined the scalars involved in the conjecture.
But these methods cannot be applied to the case of classical 
groups with even characteristic.    
\par
In this paper we take up the problem of determining the 
scalars in the case of classical groups with $p = 2$.
We show that the scalars are determined explicitly in the
case where $G = Sp_{2n}$ with $p = 2$.  We also obtain 
a somewhat weaker result for the case $SO_{2n}$ of split type,
when $p = 2$, containing the case related to the unipotent 
characters. 
The main ingredient for the proof is the theory of 
symmetric spaces over finite fields due to 
Kawanaka [K] and Lusztig [L4].  They determined the 
multiplicity of irreducible representations of $G^{F^2}$
occurring in the induced module $\Ind_{G^F}^{G^{F^2}}1$ in the
case where the center of $G$ is connected (for arbitrary 
characteristic).  Using this, 
one can determine the scalars for $G^{F^2}$ in many cases 
for a connected classical group with connected center, with
arbitrary characteristic.  
\par
On the other hand, it was 
shown in [S3] that there exists a good representatives 
of $C^F$ for a unipotent class $C$ in $Sp_{2n}$ or 
$SO_N$ for arbitrary characteristic.  
This implies that the generalized Green functions of 
$G^{F^m}$ turn out to be polynomials in $q$ (more precisely, 
rational functions in $q$ if $p = 2$) for various extension 
field $\BF_{q^m}$, and so certain values of almost characters 
are also rational functions in $q$.  
This makes it possible to apply some sort of  
 specialization  argument for the character values of 
$G^{F^m}$ for any $ m \ge 1$, and one can determine the
scalars of $G^F$ which are related to the unipotent characters, 
from the result for $G^{F^2}$.  Thus we rediscover the results in 
[S2].  But this method works also for $p = 2$,  and 
from this we can deduce the result for $G^F$.
\par
The main result of this paper was announced in the 4th
International Conference on Representation Theory, Lhasa, 2007.
\section{Lusztig's conjecture}
\para{1.1.}
Let $k$ be an algebraic closure of a finite field $\Fq$ of
characteristic $p$. 
Let $G$ be a connected reductive algebraic group defined 
over $k$.  
We fix a Borel subgroup $B$ of $G$, and a maximal torus $T$ 
contained in $B$, and a Weyl group $W = N_G(T)/T$
of $G$ with respect to $T$.  
Let $\DD G$ be the bounded derived category of constructible 
$\Ql$-sheaves on $G$, and let $\CM G$ be the full subcategory 
of $\DD G$ consisting of perverse sheaves.
Let $\CS(T)$ be the set of isomorphism classes of 
tame local systems on $T$, i.e., the 
local systems $\CL$ of rank 1 such that $\CL^{\otimes n} \simeq \Ql$ 
for some integer $n \ge 1$, invertible in $k$.
Take a local system $\CL \in \CS(T)$ such that $w^*\CL \simeq \CL$
for some $w \in W$.  Then one can construct a complex 
$K_w^{\CL} \in \DD G$ as in [L2, III, 12.1].  For each $\CL \in \CS(T)$
we denote by $\wh G_{\CL}$ the set of isomorphism classes of 
irreducible perverse sheaves $A$ on $G$ such that $A$ is a 
constituent of the $i$-th perverse cohomology sheaf $^pH^i(K_w^{\CL})$
of $K_w^{\CL}$ for any $i,w$.  The set $\wh G$ of character sheaves on 
$G$ is defined as $\wh G = \bigcup_{\CL \in \CS(T)}\wh G_{\CL}$.
\para{1.2.}
We consider the $\Fq$-structure of $G$, and assume that $G$ is 
defined over $\Fq$ with Frobenius map $F$.  We assume that $B$ and 
$T$ are both $F$-stable.
We assume further that the center of $G$ is connected.  
Let $G^*$ be the
dual group of $G$ and $T^*$ a maximal torus of $G^*$ dual to $T$.
By fixing an isomorphism $\io : k^* \simeq \BQ'/\BZ$ ($\BQ'$ is the 
subring of $\BQ$ consisting of elements whose numerator is invertible in $k$), 
we have an isomorphism 
$f: T^* \simeq \CS(T)$ (see e.g., [S1, II, 1.4, 3.1]).  
Let $W^* = N_G(T^*)/T^*$ be the Weyl group of
$G^*$.  Then $W^*$ may be identified with $W = N_G(T)/T$, 
compatible with $f$. 
$F$ acts naturally on $\CS(T)$, via $F\iv : \CL \mapsto F^*\CL$, 
and the action of $F$ on $\CS(T)$
corresponds to the action of $F\iv$ on $T^*$ via $f$.
\par
For each $s \in T^*$ such that the conjugacy class $\{ s\}$ in $G^*$
is $F$-stable, we put
\begin{align*}
W_s &= \{ w \in W^* \mid w(s) = s\}, \\
Z_s &= \{ w \in W^* \mid F(s) = w(s)\}.
\end{align*}
Then $Z_s$ is non-empty, and one can write $Z_s = z_1W_s$ for some
element $z_1 \in Z_s$.  Since the center of $G$ is connected, 
$Z_{G^*}(s)$ is connected reductive, and 
$W_s$ is a Weyl group of $Z_{G^*}(s)$. We choose $z_1$ so that 
$\g = \g_s = z_1\iv F: W_s \to W_s$ leaves invariant the set of 
simple roots of $Z_{G^*}(s)$ determined naturally from $B$ and $T$.  
\par
Similarly, for any $\CL \in \CS(T)$ such that $F^*\CL \simeq \CL$, 
we define
\begin{align*}
W_{\CL} &= \{ w \in W \mid w^*\CL \simeq \CL\}, \\
Z_{\CL} &= \{ w \in W \mid F^*\CL \simeq (w\iv)^*\CL\}.
\end{align*}
Then $W_{\CL}$ (resp. $Z_{\CL}$) is naturally identified with 
$W_s$ (resp. $Z_s$).
\para{1.3.}
Let $\Irr G^F$ be the set of irreducible characters of $G^F$.
Then $\Irr G^F$ is partitioned into a disjoint union of subsets 
$\CE(G^F, \{s\})$, where $s \in T^*$ and $\{s\}$ runs over 
all the $F$-stable semisimple classes in $G^*$.  According to
[L1], two parameter sets $X(W_s, \g)$ and $\ol X(W_s, \g)$ are
attached to $\CE(G^F, \{s\})$, and 
a non-degenerate pairing 
$\{\ ,\ \}: \ol X(W_s,\g) \times X(W_s,\g) \to \Ql$ is defined.
Here $\ol X(W_s, \g)$ is a finite set, and $X(W_s, \g)$ is an 
infinite set with a free action of the group $M$ of all roots 
of unity in $\Ql^*$.  
More precisely, there exists a set $\ol X(W_s)$ with $\g$-action, 
and a natural
map $X(W_s, \g) \to \ol X(W_s)$ whose image coincides with 
$\ol X(W_s)^{\g}$, the set of $\g$-fixed points in $\ol X(W_s)$.
In the case where $\g$ acts trivially on $W_s$, $X(W_s, \g)$ 
coincides with $\ol X(W_s, \g) \times M$.  In general, the orbits
set $X(W_s,\g)/M$ is in bijection with $\ol X(W_s)^{\g}$.
\par
Now the set $\CE(G^F, \{s\})$ is parametrized by $\ol X(W_s,\g)$.
We denote by $\r_y$ the irreducible character in $\CE(G^F,\{s\})$ 
corresponding to $y \in \ol X(W_s,\g)$.
In turn, for each $x \in X(W_s, \g)$, an almost character
$R_x$ is defined as 
\begin{equation*}
\tag{1.3.1}
R_x = (-1)^{l(z_1)}\sum_{y \in \ol X(W_s, \g)}\{ y, x\}\vD(y)\r_y,
\end{equation*}
where $\vD(y) = \pm 1$ is a certain adjustment in the case of 
exceptional groups $E_7, E_8$. 
If $c$ is the order of $\g$ on $W_s$, the almost characters $R_x$
are determined, up to a $c$-th root unity multiple, 
by the $M$-orbit of $x$ in $X(W_s,\g)$.
\para{1.4.}
It is known by [L2, V], that the set $\wh G_{\CL}$ is parametrized
by $\ol X(W_{\CL}) = \ol X(W_s)$ under the identification 
$W_{\CL} \simeq W_s$. 
For each $y \in \ol X(W_s)$, we denote by $A_y$ the 
corresponding character sheaf in $\wh G_{\CL}$.
Let $\wh G^F$ be the set of $F$-stable character sheaves, i.e., 
the set of $A \in \wh G$ such that $F^*A \simeq A$.  Then 
$\wh G^F = \bigcup_{\CL}\wh G_{\CL}^F$, where $\CL$ runs 
over the elements in $\CS(T)$ such that $(Fw)^*\CL \simeq \CL$ 
for some $w \in W$.
The set $\wh G^F_{\CL}$ is parametrized by $\ol X(W_s)^{\g}$.   
For each $A \in \wh G_{\CL}^F$, we fix an isomorphism 
$\f_A: F^*A \isom A$ as in [L2, V, 25.1].  Then $\f_A$ is 
unique up to a root of unity multiple.  We define a class function 
$\x_A = \x_{A,\f_A}$ as the characteristic function $G^F \to \Ql$
of $A$. 
In the case of classical groups, we have the following theorem, 
which is a (partial) solution to the 
Lusztig' conjecture.
\begin{theorem}[{[S1,II,  Theorem 3.2]}]
Assume that $G$ is a (connected) classical group with 
connected center. Then for each $x \in X(W_s, \g)$, there 
exists an algebraic number $\z_x$ of absolute value 1 such that
\begin{equation*}
R_x = \z_x\x_{A_{\bar x}},
\end{equation*}
where $\bar x$ is the image of $x$ under the map 
$X(W_s,\g) \to \ol X(W_s)^{\g}$.
\end{theorem}
\para{1.6.}
Assume that $G$ is a connected classical group with connected center.
Let $P$ be an $F$-stable parabolic subgroup of $G$ containing $B$,
and $L$ be an $F$-stable Levi subgroup of $P$ containing $T$, $U_P$
the unipotent radical of $P$.  
Then $W_L = N_L(T)/T$ is a Weyl subgroup of $W$, and $B_L = B \cap L$
is a Borel subgroup of $L$ containing $T$. 
Let $\wh L$ be the set of 
character sheaves on $L$.  We  assume that $\wh L_{\CL}$ contains
a cuspidal character sheaf $A_0$ for $\CL \in \CS(T)$, where 
$\CL$ is $Fw$-stable for some $w \in W$.  Then 
$A_0$ may be expressed by the intersection cohomology complex as 
$A_0 = \IC(\ol \vS, \CE)[\dim \vS ]$, where $\vS$ is the inverse image
of a conjugacy class in $\ol G = G/Z^0(G)$ under the natural map
$\pi : G \to \ol G$, and $\CE$ is a cuspidal local system on $\vS$.
The pair $(\vS, \CE)$, or its restriction on the conjugacy class, 
is called a cuspidal pair on $G$.
Then either $L$ is the maximal torus or $L$ has the same type as 
$G$, and $A_0$ is a unique cuspidal character sheaf contained 
in $\wh L_{\CL}$.  
Consider the induced complex $K = \ind_P^GA_0$ on $G$.  Then $K$ is a
semisimple perverse sheaf on $G$ whose components are contained in 
$\wh G$.  By Lemma 5.9 in [S1,I], the endomorphism algebra 
$\End_{\CM G}K$ is isomorphic to the group algebra 
$\Ql[\CW_{\CE}]$ of $\CW_{\CE}$, where 
\begin{align*}
\CW_{\CE} &= \{n \in N_G(L) \mid n\vS n\iv = \vS, 
                  \ad(n)^*\CE \simeq \CE\}/L, \\
\CZ_{\CE} &= \{ n \in N_G(L) \mid F(n\vS n\iv) = \vS, 
                    (Fn)^*\CE \simeq \CE \}/L.
\end{align*}
\par
On the other hand, if we choose a positive integer $r$ 
large enough, the set $\CE(L^{F^r}, \{s\})$ contains 
a unique cuspidal character $\d$ of $L^{F^r}$, where 
$s \in T^*$ corresponds to $\CL$ under $f$.
We define 
\begin{align*}
W_{\d} &= \{ w \in N_W(W_L) \mid wB_Lw\iv = B_L, ^w\d \simeq \d\}, \\
Z_{\d} &= \{ w \in N_W(W_L) \mid wB_Lw\iv = B_L, ^{Fw}\d \simeq \d\}.
\end{align*}
Since $\wh L_{\CL}$ contains a unique cuspidal character sheaf, 
we have $W_{\d} \simeq \CW_{\CE}, Z_{\d} \simeq \CZ_{\CE}$
by [S1, I, (5.16.1)].
Moreover there exists $w_1 \in Z_{\d}$ such that $Z_{\d} = w_1W_{\d}$
and that $\g_1 = Fw_1: W_{\d} \to W_{\d}$ gives rise to an automorphism 
of the Coxeter group $W_{\d}$. 
We denote by $(W_{\d})\wg$ the set of irreducible characters of
$W_{\d}$, and $(W_{\d})\wg\ex$ the subset of $(W_{\d})\wg$ consisting of
$\g_1$-stable characters.  Then 
$A_0$ is $Fw_1$-stable, and 
for each $E \in (W_{\d})\wg\ex$, there
exists $x_E \in X(W_s, \g)$ and $\bar x_E \in \ol X(W_s)^{\g}$ such that
$A_E = A_{\bar x_E}$ is an $F$-stable character sheaf in $\wh G_{\CL}$ 
which is a simple component of $K$ corresponding to $E \in \End_{\CM}K$.
Moreover, $\r_{x_E} \in \CE(G^{F^r}, \{s\})$ is an $F$-stable
irreducible character which is a constituent of the Harish-Chandra
induction $\Ind_{P^{F^r}}^{G^{F^r}}\d$ corresponding to 
$E \in (W_{\d})\wg\ex$, and the image of the Shintani descent 
$\Sh_{F^r/F}$ of $\r_{x_E}$ 
determines the almost character $R_{x_E}$ of $G^F$.
(For the Shintani descent, see [S1]).
\para{1.7.}
Let $L_{w_1}$ be an $F$-stable Levi subgroup twisted 
by $F(w_1)$, i.e., $L_{w_1} = \a L\a\iv$ for $\a \in G$
such that $\a\iv F(\a) = F(\dw_1)$ for a representative 
$\dw_1 \in N_G(L)$ of $w_1 \in \CZ_{\CE}$.
Then by $\ad(\a\iv) : L_{w_1} \isom L$, $\ad(\a\iv)^* A_0$
gives rise to an $F$-stable cuspidal character sheaf on 
$L_{w_1}$ which we denote by $A'_0$. 
If we fix an isomorphism $\vf_0: (F\dw_1)^*\CE \isom \CE$, 
$\vf_0$ induces an isomorphism $\vf_0^{w_1} : F^*A'_0 \isom A'_0$
on $L_{w_1}$.  We choose $\f_{A_0'}: F^*A_0' \isom A_0'$ as 
$\f_{A_0'} = \vf_0^{w_1}$. 
Then by Theorem 1.5, we have 
\begin{equation*}
\tag{1.7.1}
R_0^{L_{w_1}} = \z_0\x_{A_0'}
\end{equation*}
for some $\z_0 \in \Ql^*$ of absolute value 1, where 
$R_0^{L_{w_1}}$ is the almost character of $L_{w_1}^F$ 
corresponding to $A_0' \in \wh L_{w_1}$.
The following result was proved in [S1] in the course of the 
proof of the main theorem.
(Note that in [S1], the constants $\ve_0\xi_{A_0}$ and 
$\ve_0\xi_{A_E}$ are used. But the proof shows that these 
constants are indeed given by $\z_0 = \ve_0\xi_{A_0}$.)
\begin{lemma}[{[S1, II, Lemma 3.7]}]  
Let $\z_0$ be as in (1.7.1).  Then we have
\begin{equation*}
R_{x_E} = (-1)^{\dim \vS}\z_0\x_{A_E}
\end{equation*}
for any $E \in (W_{\d})\ex\wg$.
\end{lemma}
\para{1.9.}  Lemma 1.8 shows that the determination of 
the scalars $\z_x$ appeared in Theorem 1.5 is reduced 
to the case of cuspidal character sheaves.  We note
that it is further reduced to the case of adjoint groups.
In fact,  let $A_0$ be an $F$-stable  cuspidal character sheaf  
contained in $\wh G_{\CL}$.  Let $\pi: G \to \ol G$ be as before. Then 
$A_0$ can be written as $A_0 \simeq \CE_0\otimes \pi^*\bar A_0[d]$, 
where $d = \dim Z(G)$, and $\CE_0$ is a local system on $G$ which is 
the inverse image of $\CE_0' \in \CS(G/G\der)$ under the natural map
$G \to G\der$ ($G\der$ is the derived subgroup of $G$), and $\bar A_0$
is a cuspidal character sheaf on $\bar G$.  Since 
$\bar A_0$ is a unique cuspidal character sheaf in 
$\wh{\bar G}_{\CL}$, $\bar A_0$ is $F$-stable. Then $\pi^*\bar A_0$
is $F$-stable and so $\CE_0$ is also $F$-stable. 
Then $\f_{A_0}: F^*A_0 \isom A_0$ is given by 
$\f_{A_0} = \vf_0\otimes \pi^*\f_{\bar A_0}$, where 
$\f_{\bar A_0} : F^*\bar A_0 \isom \bar A_0$ is the map chosen for
$\bar A_0$, 
and $\vf_0$ is the pull-back of the canonical isomorphism 
$F^*\CE_0' \isom \CE_0'$. 
Hence $\x_{A_0}$ is written 
as  $\x_{A_0} = \th_0\otimes \pi^*\x_{\bar A_0}$, where 
$\pi^*\x_{\bar A_0}$ is the pull-back of 
$\x_{\bar A_0}$ under the induced map $\pi : G^F \to \bar G^F$, and 
$\th_0$ is a linear character of $G^F$ corresponding to 
$\CE_0$.  
A similar description works also for almost characters. 
Let $R_0$ (resp. $\bar R_0$) be the almost character of 
$G^F$ (resp. $\bar G^F$) corresponding to $A_0$ (resp. $\bar A_0$).
Then we have $R_0 = \th_0\otimes \pi^*\bar R_0$.  (This 
follows from the fact that if $\d$ is a cuspidal irreducible 
character of $G^{F^r}$ corresponding 
to $A_0$ for sufficiently large $r$, 
then $\d$ can be written as $\d = \th\otimes \bar\d$, 
where $\bar\d$ is a cuspidal irreducible character of $\bar G^{F^r}$
corresponding to $\bar A_0$, and $\th$ is an $F$-stable linear character 
of $G^{F^r}$, and by applying the Shintani descent on $\d$.) 
Thus $\z_0$ for $A_0$ coincides with $\z_0$ for $\bar A_0$.
\section{generalized Green functions}
\para{2.1.}
Under the setting in 1.6, we further assume that $G^F$ is of split type.
Let $L$ be as before.  
Assume that $A_0$ is a cuspidal character sheaf on $L$ of the
form $A_0 = \IC(\ol \vS, \CE)[\dim \vS]$, where $\vS = Z^0(L)\times C$
with a unipotent class $C$ in $L$ and $\CE = \Ql\boxtimes\CE'$ for a
cuspidal local system $\CE'$ on $C$.
Then $\CW_{\CE} = \CW = N_G(L)/L$.
For each $w \in \CW$, let 
$L_w$ be an $F$-stable Levi subgroup of $G$ obtained from $L$
by twisting $w$ as in 1.7, i.e., $L_w = \a L \a\iv$ with 
$\a \in G$ such that $\a\iv F(\a) = F(\dw)$ for a representative 
$\dw \in N_G(L)$ of $w$.  
Put $\vS_{w} = \a \vS \a\iv, \CE_w = \ad(\a\iv)^*\CE$, a local system 
on $\vS_w$. 
  We assume that the pair
$(C, \CE')$ is $F$-stable, and fix an isomorphism 
$\vf_0: F^*\CE' \isom \CE'$.  
Then one can construct an isomorphism $(\vf_0)_w: F^*\CE_w \isom \CE_w$
as in [L2, II, 10.6], and this induces an isomorphism 
$\vf_w : F^*K_w \isom K_w$, where $K_w$ is a complex induced from 
the pair $(\vS_w, \CE_w)$.
Note that $K_w$ is isomorphic to $\ind_P^G A_0$, with a specific 
mixed structure twisted by $w \in \CW$.
We denote by $\x_{K_w,\vf_w}$ the characteristic function of $K_w$
with respect to $\vf_w$.
\par
Since $L$ is of the same type as $G$, and $F$ is of split type, 
$\g_1 : \CW \to \CW$ is identity. Let 
$K = \ind_P^GA_0 = \bigoplus_{E \in \CW\wg}V_E\otimes A_E$ be the
decomposition of $K$  into simple components, where $A_E$ is a character
sheaf corresponding to $E \in \CW\wg$, and $V_E$ is the multiplicity
space of $A_E$ which has a natural structure of irreducible $\CW$-module
corresponding to $E$.  Then there exists a unique isomorphism 
$\f_{A_E} : F^*A_E \isom A_E$ for each $E \in \CW\wg$ such that 
\begin{equation*}
\tag{2.1.1}
\x_{K_w, \vf_w} = \sum_{E \in \CW\wg}\Tr(w, V_E)\x_{A_E}.
\end{equation*}
Let $G\uni$ be the unipotent variety of $G$, and $G\uni^F$ be the set
of $F$-fixed points in $G\uni$.  The
restriction of $\x_{K_w, \vf_w}$ on $G^F\uni$ is the generalized
Green function $Q^G_{L_w, C_w, \CE'_w, (\vf_0)'_w}$ ([L2, II, 8.3]), where 
$C_w = \a C \a\iv, \CE'_w = \ad(\a\iv)^*\CE'$, local system on $C_w$,
and $(\vf_0)_w'$ is the restriction of $(\vf_0)_w$ on $\CE_w'$. 
On the other hand, by the generalized Springer correspondence, 
for each $E \in \CW\wg$, there exists a pair $(C_1, \CE_1)$, where 
$C_1$ is a unipotent class in $G$ and $\CE_1$ is a $G$-equivariant 
simple local system on $C_1$, such that 
\begin{equation*}
\tag{2.1.2}
A_E|_{G\uni} \simeq \IC(\ol C_1, \CE_1)[\dim C_1 + \dim Z^0(L)].
\end{equation*}
Now the pair $(C_1, \CE_1)$ is $F$-stable, and 
$\f_{A_E}: F^*A_E \isom A_E$ determines an isomorphism 
$\p_{\CE_1} : F^*\CE_1 \isom \CE_1$ via (2.1.2) (cf. [L2, V, 24.2]).  
In other words,  
the choice of $\vf_0: F^*\CE' \isom\CE'$ determines $\p_{\CE_1}$.
Let $\x_{(C_1, \CE_1)} : C_1^F \to \Ql$ be the characteristic 
function of $\CE_1$ with respect to $\p_{\CE_1}$.
Lusztig ([L2, V, 24]) gave an algorithm of computing 
$\x_{A_E} = \x_{A_E, \f_{A_E}}$ 
on $G^F\uni$.  Here $\x_{A_E}|_{G\uni^F}$
is expressed in terms of a linear combination of various 
$\x_{(C'_1, \CE'_1)}$, where $(C_1', \CE_1')$ is a pair as above 
corresponding to some $E' \in \CW\wg$. Let $p_{E,E'}$ be the coefficient
of $\x_{(C_1', \CE_1')}$ in the expansion of $\x_{A_E}$.
Then $p_{E,E'}$ satisfies the following property; if we replace $F$ by 
$F^m$ for any integer $m > 0$, we obtain a similar coefficient 
$p_{E,E'}$ (with respect to $G^{F^m}$) starting from
 $\vf_0^{(m)}: (F^m)^*\CE' \isom \CE'$ induced 
naturally from $\vf_0$, which we denote by $p^{(m)}_{E,E'}$.  Then 
there exists a rational function $P_{E,E'}(x)$ such that 
$p^{(m)}_{E,E'} = P_{E,E'}(q^m)$. 
(Note : It is shown in [L2, V] that $P_{E,E'}$ turns out to be a
polynomial if $p$ is good.)
\par
Now $\x_{(C_1,\CE_1)}$ is described as follows;
take $u \in C_1^F$ and put $A_G(u) = Z_G(u)/Z_G^0(u)$.  Then 
$F$ acts naturally on $A_G(u)$, and in our setting $F$ acts trivially
on it. The set of $G^F$-conjugacy classes in $C_1^F$ is in
1:1 correspondence with the set $A_G(u)$ (note: $A_G(u)$ is abelian).
We denote by $a_u$ a representative of a $G^F$-class in $C_1^F$ 
corresponding to $a \in A_G(u)$.  On the other hand, the set of 
$G$-equivariant simple local systems on $C_1$ is in 1:1 correspondence 
with the set $A_G(u)\wg$ of irreducible characters of $A_G(u)$.
For each $\r \in A_G(u)\wg$, we defnie a function $f_{\r}$ on 
$G\uni^F$ by 
\begin{equation*}
\tag{2.1.3}
f_{\r}(v) = \begin{cases}
                \r(a)     &\quad\text{ if } v = u_a \in C_1^F, \\
                  0       &\quad\text{ if } v \notin C_1^F
             \end{cases}
\end{equation*}
for $v \in G\uni^F$.
Let $\r \in A_G(u)\wg$ be the character corresponding to $\CE_1$. 
Then there exists $\e_E \in \Ql^*$ of absolute value 1
such that 
\begin{equation*}
\tag{2.1.4}
\x_{(C_1,\CE_1)} = \e_E f_{\r}.
\end{equation*}
Note that $\e_E$ depends on the choice
of $\vf_0: F^*\CE' \isom \CE'$ and on the choice of $u \in C_1^F$.
We have the following theorem.
\begin{theorem}[{[S3]}]  
Let $G$ be a classical group, simple modulo center.  
Assume that the derived subgroup of $G$ does not contain 
the Spin group.  Further assume that $G$ is of split type.
Then for each unipotent class $C_1$
in $G$, there exists $u \in C_1^F$ (called a split unipotent element)
satisfying the following; Let $(C, \CE')$ be the pair in $L$ as in 2.1 
and $u_0 \in C^F$ be a split element.
Choose $\vf_0: F^*\CE' \isom \CE'$ so that the isomorphism 
$(\vf_0)_{u_0} : \CE'_{u_0} \to \CE'_{u_0}$ induced on the stalk 
$\CE'_{u_0}$ of $\CE'$ at $u_0$ is identity.
Choose a split element $u \in C_1^F$ for defining $f_{\r}$ in 
(2.1.3).  Then $\e_E = 1$ for any $E \in \CW\wg$.  
\end{theorem}
\para{2.3.}
Returning to the setting in 2.1, we choose a split 
element $u \in C_1^F$ for each unipotent class $C_1$ of $G$.
   Then $u \in C_1^{F^m}$ for any integer
$m > 0$, (in fact, it is a split element with respect to $G^{F^m}$), 
and we choose $u^{(m)}_a \in C_1^{F^m}$,  
a representative of $G^{F^m}$-class in $C_1^{F^m}$ for each 
$a \in A_G(u)$.
For later discussion, we prepare a notation.   
Assume given a family of functions 
$h = \{ h^{(m)}\}_{m>0}$, where $h^{(m)}$ is a class function on $G^{F^m}\uni$.
Then we say that $h$ is a rational function in $q$ if there exists 
a rational function $H_{C_1,a}(x)$ for each pair $(C_1,a)$ 
such that $h^{(m)}(u_a^{(m)}) = H_{C_1,a}(q^m)$. 
\par
For each $E \in \CW\wg$, we have an isomorphism 
$\f^{(m)}_{A_E}: (F^m)^*A_E \isom A_E$, and one can define 
a function $\x^{(m)}_{A_E} = \x_{A_E, \f^{(m)}_{A_E}}$ on $G^{F^m}$.
Then in view of Theorem 2.2, Lusztig's algorithm implies that 
$\{ \x^{(m)}_{A_E}|_{G^{F^m}\uni}\}_{m >0}$ 
is a rational function in $q$.
\par
Thus, by (2.1.1) we have the following corollary.
\begin{cor}  
Assume that $G$ is as in Theorem 2.2, and that $G^F$ is of split 
type.  The generalized Green function 
$Q^G_{L_w, C_w, \CE'_w, (\vf_0)'_w}$ can be expressed as a rational
function in $q$. 
\end{cor}
\para{2.5.}
More generally, if there exists a family of values 
$h = \{ h^{(m)} \in \Ql\}_{m > 0}$ such that $h^{(m)} = H(q^m)$ for
some rational function $H(x)$, we say that $h$ is a rational function 
in $q$. 
\section{Cuspidal character sheaves}
\para{3.1.}
Let $G$ be an adjoint simple group of classical type.
We assume that $G$ is of split type over $\Fq$.
Let $\wh G^0$ be the set of cuspidal character sheaves on $G$.
$A \in \wh G^0$ is given in the form $A = \IC(\ol C, \CE)[\dim C]$, where
$C$ is a conjugacy class in $G$, and $\CE$ is a simple $G$-equivariant 
local system on $C$. 
We shall describe the cuspidal character sheaves on $G$
(cf. [L2, V, 22.2, 23.2], see also [S2, 6.6]) and their mixed structures.
\par\medskip
(a) \ $G = PSp_{2n}$ ($n \ge 1$) with $p: $ odd.
$\wh G^0$ is empty if $n$ is even.  Assume that $n$ is odd.
Then for each pair $(N_1, N_2)$ such that $N_i = d_i^2 + d_i$
for some integers $d_i \ge 0$ and that $n = N_1 + N_2$, one 
can associate cuspidal character sheaves on $G$ as follows.   
Let $C$ be a conjugacy class 
of $g = su = us$, where $s$ is a semisimple element of $G$ 
such that $Z^0_G(s)$ is isomorphic 
to $H = (Sp_{2N_1} \times Sp_{2N_2})/\{ \pm 1\}$, and 
$u$ is a unipotent element of $Z^0_G(s) \simeq H$ such that the  
unipotent class $C_0$ containing $u$ gives a unique cuspidal pair 
$(C_0, \CE_0)$ with unipotent support of $H$.  
Here $(C_0, \CE_0)$ is described as follows.
There exists a  cuspidal pair $(C_i, \CE_i)$ for $Sp_{2N_i}$
such that $C_0 = C_1\times C_2$ and $\CE_0 = \CE_1\boxtimes \CE_2$.
Choose $u = (u_1, u_2) \in C_0$ such that $u_i \in C_i$.  
Let $\r_i \in A_{H_i}(u_i)\wg$ corresponding to $\CE_i$, where 
$H_i = Sp_{2N_i}$.  Then 
$\r_1\boxtimes\r_2 \in (A_{H_1}(u_1) \times A_{H_2}(u_2))\wg$ 
factors through $A_H(u)$ and defines an irreducible character $\r_0$
of $A_H(u)$ corresponding to $\CE_0$.
\par
Now assume that $N_1 \ne N_2$.  Then $Z_G(s)$ is connected, and so
$A_G(g) = A_H(u)$, and $\r_0$ gives an irreducible 
character $\r \in A_G(g)\wg$ which determines a local system $\CE$ 
on $C$, and we denote by $A_{N_1,N_2}$ the character sheaf corresponding
to $(C,\CE)$. 
Next assume that $N_1 = N_2$.  Then $A_G(s) \simeq \BZ/2\BZ$ and 
$A_H(u)$ is a subgroup of $A_G(g)$ of index 2.   
We have $\Ind_{A_H(u)}^{A_G(g)}\r_0 = \r + \r'$, where 
$\r, \r'$ are linear characters of $A_G(u)$.
If we write $\CE, \CE'$ the simple local system on $C$ corresponding 
to $\r,\r'$, then the pairs $(C, \CE), (C, \CE')$ are both cuspidal
pairs of $G$.  We denote by $A_{N_1, N_2}, A'_{N_1, N_2}$ the 
cuspidal character sheaves on $G$ corresponding to 
$(C,\CE), (C,\CE')$, respectively.  
The set $\wh G^0$ consists of these elements.
\par
We shall fix a mixed structure on $(C,\CE)$.    
Since $s \in G^F$, $H$ is $F$-stable, and so $(C_0, \CE_0)$ 
is also $F$-stable. Choose $u = (u_1, u_2) \in C_0^F$ such that 
$u_i$ are split elements in $Sp_{2N_i}$, and fix $s \in T^F$ 
appropriately.
We choose $\vf_0: F^*\CE \isom \CE$ so that the induced 
isomorphism $(\vf_0)_g: \CE_g \to \CE_g$ on the stalk $\CE_g$ at
$g$ is identity.  Then $\vf_0$ induces an isomorphism 
$\vf: F^*A_{N_1, N_2} \isom A_{N_1,N_2}$.  We define 
$\f_A = \f_{A_{N_1, N_2}}$ by $\f_A = q^{(\dim G - \dim C)/2}\vf$.
A similar construction is applied also for $(C,\CE')$.   
\par\medskip
(b) \ $G = PSO_m$ ($m \ge 3$) with $p:$ odd.
$\wh G^0$ is empty unless $m$ is either odd or divisible by 8.
Note that $PSO_m = SO_m$ if $m$ is odd. 
To each pair $(N_1, N_2)$ such that $N_i = d_i^2$ for some $d_i \ge 1$
 and that 
$m = N_1 + N_2$, one can associate cuspidal character sheaves
$A$ associated to $(C, \CE)$ as follows. 
Let $C$ be the conjugacy class of $G$ containing $g = su = us$, where
$s$ is a semisimple element such that $H = Z_G^0(s)$ is isomorphic to
$SO_{N_1}\times SO_{N_2}$ if $m$ is odd, and to
$(SO_{N_1}\times SO_{N_2})/\{ \pm 1\}$ if $m$ is even, 
and $u$ is a unipotent element in $Z_G^0(s) \simeq H$ 
such that the unipotent class
$C_0$ containing $u$ gives a unique cuspidal pair 
$(C_0, \CE_0)$ with unipotent support on $H$.
Here $C_0 = C_1\times C_2$, $\CE_0 \simeq \CE_1\boxtimes \CE_2$
with the cuspidal pair $(C_i, \CE_i)$ on $SO_{N_i}$.
Choose $u = (u_1, u_2) \in C_0$ such that $u_i \in C_i$.  
Let $\r_i \in A_{H_i}(u_i)\wg$ corresponding to $\CE_i$, where
$H_i = SO_{N_i}$.
Then $\r_1\boxtimes\r_2 \in (A_{H_1}(u_1) \times A_{H_2}(u_2))\wg$
factors through $A_H(u)$ and gives an irreducible character 
$\r_0 \in A_H(u)\wg$ corresponding to $\CE_0$.
Depending on the structure of $A_G(s)$, the three cases occur.
\par
(i) \ The case where $N_1 = 0$ or $N_2 = 0$.   In this case, 
$Z_G(s)$ is connected and so $A_G(u) \simeq A_H(u)$.
$\r_0$ gives $\r \in A_G(g)\wg$, which determines a local 
system $\CE$ on $C$, and $(C,\CE)$ corresponds to the cuspidal
character sheaf $A_{N_1, N_2}$.
\par
(ii) \ The case where $N_1 > 0, N_2 > 0, N_1 \ne N_2$.
Then $A_G(s) \simeq \BZ/2\BZ$, and $A_H(u)$ is regarded as 
an index 2 subgroup
of $A_G(g)$.  We have $\Ind_{A_H(u)}^{A_G(g)}\r_0 = \r + \r'$ for 
$\r,\r' \in A_G(g)\wg$.
If we write $\CE, \CE' $ the simple local system corresponding to
$\r,\r'$, $(C, \CE), (C, \CE')$ are both cuspidal pairs for $G$.
We denote by $A_{N_1, N_2}, A'_{N_1, N_2}$ the cuspidal character 
sheaves corresponding to them.
\par
(iii) \ The case where $N_1 = N_2$. In this case 
$A_G(s) \simeq \BZ/2\BZ \times \BZ/2\BZ$ and so 
$A_G(g)/A_H(u) \simeq \BZ/2\BZ\times \BZ/2\BZ$. 
$\Ind_{A_H(u)}^{A_G(g)}\r_0$ decomposes into 
4 irreducible (linear) characters, $\r, \r', \r'', \r'''$ of 
$A_G(g)$.  Correspondingly, we have simple local systems 
$\CE, \CE', \CE'', \CE'''$ on $C$, and all of them give
cuspidal pairs on $G$. 
We denote by $A_{N_1, N_2}$,$ A'_{N_1, N_2}$, 
$A''_{N_1, N_2}$, $A'''_{N_1, N_2}$ 
the cuspidal character sheaves 
corresponding to them.
\par
All of the above three cases give the set $\wh G^0$.
We shall fix a mixed structure on cuspial character sheaves.
Since $s \in G^F$, $H$ is $F$-stable, and so $(C_0,\CE_0)$
is $F$-stable.
Take $u = (u_1, u_2) \in C_0^F$ such that $u_i$
are split elements in $SO_{N_i}$, and fix $s \in T^F$.  We choose 
$\vf_0: F^*\CE \isom \CE$ so that the induced isomorphism 
$(\vf_0)_g: \CE_g \to \CE_g$ on the stalk $\CE_g$ at $g$ is
identity.  
$\vf_0$ induces an isomorphism $\vf: F^*A \isom A$ for 
$A = A_{N_1,N_2}$. We define $\f_{A}$ by 
$\f_A = q^{(\dim G - \dim C)/2}\vf$.  We define similarly
for $A'_{N_1,N_2}, A''_{N_1,N_2}, A'''_{N_1,N_2}$.
\par\medskip
(c) $G = Sp_{2n}$ ($n \ge 1$) with $p = 2$.
$\wh G^0$ is empty unless $n = d^2 + d$ for some $d \ge 1$.
Assume that $n = d^2 + d$.  Then $G$ contains a unique cuspidal 
pair $(C,\CE)$.  The set $\wh G^0$ consists of a single 
character sheaf $A$ associated to $(C,\CE)$.
We fix a mixed structure of $A$.  $C$ is an $F$-stable unipotent
class of $G$, and we take a split element $u \in C^F$.
We fix an isomorphism $\vf_0: F^*\CE \isom \CE$ so that 
the induced isomorphism $(\vf_0)_u: \CE_u \to \CE_u$ on the 
stalk $\CE_u$ of $\CE$ at $u$ is identity. 
$\vf_0$ induces $\vf : F^*A \isom A$.  We define $\f_A$ by 
$\f_A = q^{(\dim G - \dim C)/2}\vf$.
\par\medskip
(d) $G = SO_{2n}$ with $p = 2$ ($n \ge 1$). 
$\wh G^0$ is empty unless $n = 4d^2$ for some $d \ge 1$.
Assume that $n = 4d^2$. Then $G$ contains a unique cuspidal 
pair $(C,\CE)$.  The set $\wh G^0$ consists of a single 
character sheaf $A$ associated to $(C,\CE)$.
$C$ is an $F$-stable unipotent
class of $G$, and we take a split element $u \in C^F$.
We define $\f_A$ in a similar way as in the case (c), by 
$\f_A = q^{(\dim G - \dim C)/2}\vf$.
\par\medskip
We show the following lemma.
\begin{lemma}  
Let $\r$ be the irreducible character of $A_G(g)$ corresponding 
to the local system $\CE$ on $C$, as in 3.1.  Then 
$\r$ is a linear character such that $\r^2 = 1$.
A similar fact holds also for $\r', \r'', \r'''$ if there exists any. 
\end{lemma}
\begin{proof}
Let $\r$ be one of the characters $\r, \r', \r'', \r'''$ if 
there exists any.  
It is enough to show that $\r(a^2) = 1$ for any $a \in A_G(g)$.
By investigating the structure of 
$A_G(g)$, we see that $A_G(g)$ is an elementary abelian 2-group
if $N_1 \ne N_2$.  Thus in this case, $\r(a^2) = 1$.
We assume that $N_1 = N_2$.  Then $p$ is odd, and $G$ is $PSp_{2n}$
or $PSO_{2n}$.  Assume that $G = PSp_{2n}$.  We have 
$A_G(g) \simeq \lp \s \rp \ltimes A_H(u)$, where 
$\s$ is an element of order 2 permuting two factors of $A_H(u)$. 
In this case $a \in A_G(g)$ is of order 2 or 4.  
If $a$ has order 2, there is nothing to prove. Assume that 
$a$ has order 4.  Then we have $a^2 \in A_H(u)$.  
Put $\th = \Ind_{A_H(u)}^{A_G(g)}\r_0$. 
Since $A_H(u)$ is an elementary abelian 
2-group, and $\r_0$ is $\s$-stable, 
we see that $\th(a^2) = |A_G(g)|/|A_H(u)|$.
This shows that $\r(a^2) = 1$ for any irreducible factor $\r$ of 
$\th$.  Next assume that $G = SO_{2n}$.  In this case, 
$A_G(g) \simeq \lp\s\rp\ltimes \wt A_H(u)$, where $\wt A_H(u)$ is
an elementary abelian 2-group containing $A_H(u)$ as an index 2
subgroup, and $\s$ is an element of order 2 acting on $\wt A_H(u)$.
$\s$ stabilizes $A_H(u)$ permuting their two factors, and  $\r_0$ 
is $\s$-stable.  Thus a similar argument shows that $\r(a^2) = 1$.
The lemma is proved.
\end{proof}
\section{Symbols and unipotent characters}
\para{4.1.}
Irreducible characters contained in $\CE(G^F, \{1\})$ are called 
unipotent characters. In the case of classical groups, unipotent
characters are parametrized by a combinatorial object called symbols.
In this section, we review unipotent characters of classical groups.
\par
Let $G$ be a classical group over $\Fq$ of type $B_n, C_n$ or $D_n$.  
We assume that $G^F$ is of split type if $G$ is of type $D_n$.
The set of unipotent characters of $G^F$
is parameterized by symbols. A symbol is an (unordered) pair 
$\binom{S}{T}$ of finite
subsets of $\{0,1,2,\dots\}$ modulo the shift operation 
$\binom{S}{T} \sim \binom{S'}{T'}$ with $S' = \{0\} \cup (S+1)$, 
$T' = \{ 0\} \cup (T+1)$.  The rank of a symbol $\vL = \binom{S}{T}$
is defined by 
\begin{equation*}
r(\vL) = \sum_{\la \in S}\la + \sum_{\mu \in T}\mu
- \bigg[ \bigg( \frac{|S| + |T| -1}{2}\bigg)^2\bigg], 
\end{equation*}
where $[z]$ denotes the largest 
integer which does not exceed $z$. The defect $d(\vL)$ of $\vL$ is
defined by the absolute value of $|S| - |T|$.  
The rank and the defect are independent of the shift operation. 
\par
For each integer $d \ge 0$, 
we denote by $\F_n^d$ the 
set of symbols of rank $n$ and defect $d$. 
In the case where $\vL = \binom{S}{T}$ is defect 0, $\vL$
is said to be degenerate if $S = T$, and is said to be 
non-degenerate otherwise. 
We denote by $\wt\F^0_n$ the set of symbols of rank $n$ and 
defect 0, where the degenerate symbols are counted twice. 
We put
\begin{equation*}
\F_n = \coprod_{d : \odd}\F^d_n, \qquad
\F^+_n = \wt\F^0_n \coprod 
      \big(\coprod_{d \equiv 0 \pmod 4}\F^d_n\big).
\end{equation*}
Then the unipotent characters of $G^F$ of type $B_n$ or $C_n$ 
(resp. $D_n$ of split type) are parametrized by $\F_n$ 
(resp. $\F^+_n$).  In the notation of 1.3, $W_s = W$ and 
$\g = 1$ since $F$ is of split type, and $\F_n$ or 
$\F_n^+$ is nothing but $\ol X(W_s, \g) = \ol X(W, 1)$.  
We denote by $\r_{\vL}$ the unipotent
character of $G^F$ corresponding to $\vL \in \F_n$ or $\F^+_n$.  
The unipotent cuspidal character exists if and only if $n = d^2 + d$
(resp. $n = 4d^2$) for some integer 
$d \ge 1$ if $G$ is of type $B_n$ or $C_n$ (resp. $D_n$).  
In these cases, the symbol $\vL_c$ (the cuspidal symbol) corresponding 
to the (unique) cuspidal unipotent character is given as follows. 
\begin{equation*}
\begin{aligned}
\vL_c &= \binom{ 0,1,2, \dots, 2d}{-} \in \F^{2d+1}_n 
&\quad (&G : \text{ type $B_n$ or $C_n$}, n = d^2 + d), \\
\vL_c &= \binom{0,1,2, \dots, 4d-1}{-} \in \F^{4d}_n
&\quad (&G : \text{ type $D_n$}, n = 4d^2).
\end{aligned}
\end{equation*}
\para{4.2.}
We introduce a notion of families in $\F_n$ or $\F_n^+$. 
Two symbols $\vL, \vL'$ belong to the same family if $\vL, \vL'$ 
are represented by $\binom{S}{T}, \binom{S'}{T'}$ such that 
$S\cup T = S'\cup T'$ and that $S\cap T = S' \cap T'$.
Families give a partition of $\F_n$ or $\F^+_n$.
A symbol $\vL \in \F_n$ of defect 1 is called a special symbol if 
$\vL = \binom{S}{T}$ with $S = \{ a_0, a_1, \dots, a_m \}$ 
and $T = \{ b_1, \dots, b_m \}$  such that 
$a_0 \le b_1 \le a_1 \le \cdots \le b_m \le a_m$.
Similarly, a symbol $\vL \in \F^+_n$ of defect 0 is called 
a special symbol if 
$\vL = \binom{S}{T}$ with 
$S = \{ a_1, \dots, a_m \}$, $T = \{ b_1, \dots, b_m\}$
such that $a_1 \le b_1 \le \cdots \le a_m \le b_m$.
Each family contains a unique special symbol.
Let $\CF$ be a (non-degenerate) family.  Then any symbol $\vL \in \CF$ 
can be expressed as 
\begin{equation*}
\vL = \vL_M = \binom{Z_2 \coprod (Z_1 - M)}{Z_2 \coprod M},
\end{equation*}
for some $M$, where $Z_1, Z_2$ are determined by $\CF$; 
 $Z_2$ is the set of elements which appear in both rows of 
$\vL$, $Z_1$ is the set of singles in $\vL$, and $M$ is a subset 
of $Z_1$.  The map $M \mapsto \vL_M$ gives a bijective 
correspondence between the set of subsets $M$ of $Z_1$ such that
$|M| \equiv d_1 \pmod 2$ and $\CF$, where $|Z_1| = 2d_1+1$ 
(resp. $|Z_1| = 2d_1$) if $\CF \subset \F_n$ 
(resp. $\CF \subset \F_n^+$) for some integer $d_1 \ge 1$. 
(In the case of $\CF \subset \F_n^+$, we further assume that 
the smallest element in $M$ is bigger that that of $Z_1 - M$.)
In particular, the special symbol in $\CF$ can be written 
as $\vL_{M_0}$ for some $M_0 \subset Z_1$ such that 
$|M_0| = d_1$. 
For $M \in Z_1$, put $M^{\sharp} = M_0 \cup M - M_0 \cap M$.
We define 
a pairing $\{ \ , \ \} : \CF \times \CF \to \BQ$
by 
\begin{equation*}
\tag{4.2.1}
\{ \vL_M, \vL_{M'}\} = 
   \frac{1}{2^{f}}(-1)^{|M^{\sharp} \cap {M'}^{\sharp}|}, 
\end{equation*}
where $f = d_1$ (resp. $f = d_1-1$) if $\CF \subset \F_n$
(resp. $\CF \subset \F_n^+$).
We extend this pairing to the pairing on $\F_n$ or $\F_n^+$ by requiring 
that $\CF$ and $\CF'$ are orthogonal if $\CF \ne \CF'$, which we
denote by the same symbol.
Note that the pairing $\{ \ ,\ \}$ on $\F_n$ or $\F_n^+$ coincides 
with the pairing $\{ \ ,\ \}$ on $\ol X(W,1)$ given in 1.3. Hence, 
for each $\vL \in \CF$, the almost character $R_{\vL}$ is given as 
\begin{equation*}
\tag{4.2.2}
R_{\vL} = \sum_{\vL' \in X_n}\{ \vL, \vL'\}\r_{\vL'},
\end{equation*}
where $X_n = \F_n$ or $\F_n^+$ according to the cases $G$ is of
type $B_n$ or $C_n$, or $G$ is of type $D_n$.
By the property of the pairing $\{ \ , \ \}$, one can also write,
for each $\vL \in \CF$, 
\begin{equation*}
\tag{4.2.3}
\r_{\vL} = \sum_{\vL' \in X_n}\{ \vL, \vL'\}R_{\vL'}.
\end{equation*}
\para{4.3.}
Assume that $G^F$ contains a cuspidal unipotent character, and 
denote by $\CF_c$ the family containing the cuspidal 
symbol $\vL_c$.  Then the special symbol $\vL_0$ contained 
in $\CF_c$ is given as follows. 
\begin{equation*}
\begin{aligned}
\vL_0 &= \binom{0,2,4, \dots, 2d}{1,3,\dots,2d-1 }
&\qquad (&G: \text{type $B_n$ or $C_n$}, n = d^2 +d ), \\
\vL_0 &= \binom{0,2, \dots, 4d-2}{1,3,\dots,4d-1 }
&\qquad (&G: \text{ type $D_n$}, n = 4d^2). 
\end{aligned}
\end{equation*}
In the case where $G$ is of type $B_n$ or $C_n$, 
we have $Z_1 = \{ 0,1, \dots, 2d\}$ and 
$M_0 = \{ 1,3, \dots, 2d-1\}$.
In the case where $G$ is of type $D_n$, 
we have
$Z_1 = \{ 0,1, \dots, 4d-1\}$ and $M_0 = \{ 1,3, \dots, 4d-1\}$.
In both cases, $\vL_c$ is given by $\vL_c = \vL_M$ with 
$M = \emptyset$.
We denote by $R_0 = R_{\vL_c}$ the cuspidal almost character.
\para{4.4.}
Let $A_0$ be the cuspidal character sheaf of $G$ 
corresponding to $R_0 = R_{\vL_c}$.
Let $(C,\CE)$ be the cuspidal pair corresponding to $A_0$.  Then 
it is contained in the list in 3.1. If $p = 2$, it is 
uniquely determined since $\wh G^0$ consits of a single 
element. In the case where $p \ne 2$, the explicit 
correspondence is known 
by [S2, Prop. 6.7], see also [L3]. 
The conjugacy class is given as follows  
(though we don't need it in later discussions).  
We use the notation in 3.1.  Let $g = su = us \in C$ with 
$H = Z^0_G(s)$.  Assume that $G = PSp_{2n}$ with 
$n = d^2 +d$.  Then $H$ is isogeneous to 
$Sp_{d^2+d} \times Sp_{d^2+d}$. 
Assume that $G = PSO_{2n+1}$ with $n = d^2+d$.  
Then $H$ is isogenous to $SO_{(d+1)^2}\times SO_{d^2}$.
Assume that $G = PSO_{2n}$ with $n = 4d^2$.  Then
$H$ is isogeneous to $SO_{4d^2}\times SO_{4d^2}$.

\section{Symmetric space over finite fields}
In this section, we apply the theory of symmetric space over 
finite fields to the problem of determining the scalars $\z_x$
occurring in Lusztig's conjecture (Theorem 1.5)] for $G^{F^2}$.
\para{5.1.}
Let $G$ be a connected reductive group over a finite field 
$\Fq$ with Frobenius map $F$.
We consider the symmetric space $G^{F^2}/G^F$.
For a class function $f$ on $G^{F^2}$, we define $m_2(f)$ by 
\begin{equation*}
m_2(f) = \lp \Ind_{G^F}^{G^{F^2}}1, f\rp_{G^{F^2}}
       =  \frac{1}{|G^F|}\sum_{x \in G^F}f(x).
\end{equation*}
In the case where $G$ has a connected center, $m_2(\r)$ is 
determined by Kawanaka [K], Lusztig [L4] 
for any irreducible character $\r$ of $G^{F^2}$.
\para{5.2.}
Let $C$ be an $F$-stable  conjugacy class in $G$.  Take 
$x \in C^{F}$ and let $A_G(x)$ be the 
component group of $Z_G(x)$ as before.  $F$ acts naturally on $A_G(x)$.
We assume that $F$ acts trivially on $A_G(x)$.  
Then the set of $G^{F}$-conjugacy classes in $C^{F}$ is 
in bijection with the set $A_G(x)\ssim$ of conjugacy classes in 
$A_G(x)$.  The correspondence is given as follows;  for each 
$a \in A_G(x)$, take a representative $\da \in Z_G(x)$.  
There exists $h_a \in G$ such that $h_a\iv F(h_a) = \da$. 
Then $x_a = h_a x h_a\iv$ in contained in $C^F$, and 
the set $\{ x_a \mid a \in A_G(x)\ssim \}$ gives a complete 
set of representatives of the $G^{F}$-conjugacy classes in 
$C^{F}$.  
\par
The above description works also for the case of $C^{F^2}$. 
We denote by $\{ y_a \mid a \in A_G(x)\ssim \}$ the set of 
$G^{F^2}$-conjugacy classes in $C^{F^2}$.   
We define a class function $f_{\t}$ on $G^{F^2}$ for each 
$\t \in A_G(x)\wg$ as follows.
\begin{equation*}
\tag{5.2.1}
f_{\t}(g) = \begin{cases}
                \t(a) &\quad
  \text{ if $g$ is $G^{F^2}$-conjugate to $y_a$},  \\
                 0    &\quad\text{ if } g \notin C^{F^2}.
            \end{cases}
\end{equation*}
We have the following lemma.
\begin{lemma}  
Let $\t \in A_G(x)\wg$ be a linear character such that 
$\t^2 = 1$.  Then we have
\begin{equation*}
m_2(f_{\t}) = |C^F|/|G^F|.
\end{equation*}
\end{lemma}
\begin{proof}
Take $g \in C^F$.  Then $g$ is $G^F$-conjugate to an $x_a \in C^F$
for some $a \in A_G(x)$.  Then there exists  
$h \in G$ such that $g = hxh\iv$ and that $h\iv F(h) = \da \in Z_G(x)$.
Since $F$ acts trivially on $A_G(x)$, we may choose $\da \in Z_G(x)^F$.
We have
\begin{equation*}
h\iv F^2(h) = h\iv F(h)\cdot F(h\iv F(h)) = \da F(\da) = \da^2.
\end{equation*}
It follows that any element $g \in C^F$ is 
$G^{F^2}$-conjugate to $g_{a^2}$ for some $a \in A_G(x)$. Hence we have
$f_{\t}(g) = \t(a^2) = 1$ by our assumption.
We have
\begin{equation*}
m_2(f_{\t}) = \frac{1}{|G^F|}\sum_{g \in C^F}f_{\t}(g) 
            = |C^F|/|G^F|
\end{equation*}
as asserted.  The lemma is proved.
\end{proof}
\para{5.4.}
Assume that $G$ is as in 3.1, and we use the notation there.  
Let
$A = \IC(\ol C, \CE)[\dim C]$ be the cuspidal 
character sheaf on $G$. (Here $\CE$ represents one of the simple 
local systems $\CE, \CE', \dots$ on $C$ if there exist more than one).  
Then $A$ is $F$-stable.  Let $\r \in A_G(g)\wg$ be the irreducible
character corresponding to $\CE$.
Note that $F$ acts trivially on $A_G(g)$. 
We consider the class function $f_{\r}$ on $G^{F^2}\uni$ defined as in 
(5.2.1) for $\t = \r$.  
We show that 
\begin{equation*}
\tag{5.4.1}
m_2(f_{\r}) = q^{-(\dim G - \dim C)}.
\end{equation*}
In fact thanks to Lemma 3.2, one can apply Lemma 5.3 and 
we have 
$m_2(f_{\r}) = |C^F|/|G^F|$.  Since $F$ acts trivially on 
$A_G(g)$, $C^F$ splits into 
several $G^F$-conjugacy classes, which are parametrized by 
$A_G(g)\ssim$.  Let $a_1, \dots, a_r$ be 
the representatives of the conjugacy classes in $A_G(g)$, and 
let $\ZC_i$ the $G^F$-conjugacy classes in $C^F$ 
corresponding to $a_i$.  We choose $g_i \in \ZC_i$. 
Since $|Z_G(g_i)^F| = |A_A(a_i)||Z^0_G(g_i)^F|$, where 
$A = A_G(g)$, we have
\begin{equation*}
|C^F|/|G^F| = \sum_{i=1}^r|\ZC_i|/|G^F| 
            = \sum_{i = 1}^r|Z_A(a_i)|\iv |Z^0_G(g_i)^F|\iv.
\end{equation*}
Here $Z^0_G(g_i)^F \simeq Z^0_{H_i}(u_i)^F$, where 
$g_i = s_iu_i = u_is_i$ and $H_i = Z^0_G(s_i)$.  
Note that $Z^0_{H_i}(u_i) \simeq Z_H^0(u)$, and since 
$u \in C_0$, where $(C_0, \CE_0)$ is a cuspial pair with 
unipotent support, 
it is known by [L2,I, Prop.3.12] that $Z_H^0(u)$ is a unipotent group. 
It follows that 
\begin{equation*}
|Z^0_G(g_i)^F| = |Z^0_{H_i}(u_i)^F| = q^{\dim H - \dim C_0}
               = q^{\dim G - \dim C}.
\end{equation*}
Hence we have
\begin{equation*}
|C^F|/|G^F| = q^{-(\dim G - \dim C)}\sum_{i=1}^r|Z_A(a_i)|\iv 
            = q^{-(\dim G - \dim C)}.
\end{equation*}
This proves (5.4.1).
\par
Let $\vf_0: (F^2)^*\CE \isom \CE$, $\vf : (F^2)^*A \isom A$
and $\f_A: (F^2)^*A \isom A$ be as in 3.1, but replacing $F$ by 
$F^2$. 
Then the characteristic function $\x_{\CE, \vf_0}$ on $C^{F^2}$ 
coincides with $f_{\r}|_{C^{F^2}}$.
Since $A$ is clean, the function $\x_{A, \vf}$ coincides with 
$f_{\r}|_{\ol C^{F^2}}$.  It follows that $\x_A$ coincides with 
$q^{\dim G - \dim C}f_{\r}$.
This implies that 
\begin{lemma}  
Let $A$ be a cuspidal character sheaf of $G$. 
Then we have 
\begin{equation*}
m_2(\x_A) = 1.
\end{equation*}
\end{lemma}
\para{5.6.}
Let $s \in T^*$ be such that the class $\{ s\}$ is $F^2$-stable.
Then there exists $s_0 \in \{ s\}$ such that $F^2(s_0) = s_0$.
Let $H = Z_G^0(s_0)$.  Then $H$ is an $F^2$-stable reductive subgroup
of $G$.  It is known, since the center of $G$ is connected, that 
there exists a natural bijection 
$\CE(G^{F^2}, \{s\}) \lra \CE(H^{F^2}, \{1\})$, $\r \lra \r\uni$.  
Concerning the values of $m_2(\r)$ and $m_2(\r\uni)$, the following 
result is known. (For unipotent characters, we follow the notation in 
Section 4.)
\begin{theorem}[{[K], [L4]}]
Let $G$ be a connected classical group with connected center.
Then
\begin{enumerate}
\item
If there does not exist $s_0 \in \{ s\}^{F^2}$ 
such that $F(s_0) = s_0\iv$, then $m_2(\r) = 0$ for any 
$\r \in \CE(G^{F^2}, \{s\})$.  
If there exists such $s_0$, then under the notation of 
5.6, $m_2(\r) = m_2(\r\uni)$ for any $\r \in \CE(G^{F^2}, \{s\})$.
\item
Assume that $G$ is of type $B_n$ or $C_n$. Let $\CF$ be a family in 
$\F_n$ such that $|Z_1| = 2d_1+1$  (cf. 4.2).  
Then we have 
\begin{equation*}
m_2(\r_{\vL}) = \begin{cases}
             2^{d_1} &\quad\text{ if $\vL$ is special}, \\
              0      &\quad\text{ otherwise.}
                 \end{cases} 
\end{equation*}
\item
Assume that $G$ is of type $D_n$. Let $\CF$ be a non-degenerate 
family in $\F_n^+$ such that $|Z_1| = 2d_1$ (cf. 4.2).  Then we have
\begin{equation*}
m_2(\r_{\vL}) = \begin{cases}
                  2^{d_1-1}  &\quad\text{ if $\vL$ is special}, \\
                  0          &\quad\text{ otherwise}.
                \end{cases}
\end{equation*}
If $\CF = \{ \vL, \vL'\}$ is a degenerate family, then we have
\begin{equation*}
m_2(\r_{\vL}) = m_2(\r_{\vL'}) = \begin{cases}
                              1 &\quad\text{ if $F$ is of split type},  \\
                              0 &\quad\text{ otherwise.}
                                 \end{cases}
\end{equation*}
\end{enumerate}
\end{theorem}

\par
In view of (4.2.2), we have the following  corollary.
\begin{cor}[{[L4]}] 
Let $G$ be a classical group of split type.
Assume that $\CF$ is a non-degenerate family.  
Then for any $\vL \in \CF$, we have $m_2(R_{\vL}) = 1$.
\end{cor}
\para{5.9.}
Let $A$ be as in 5.4 and assume that $A \in \wh G_{\CL}$.
Let $s \in T^*$ be such that the class $\{ s\}$ corresponds 
to $\CL$ via $f$ in 1.2. 
Then $s^2 = 1$.
Let $R_0$ be the almost character of $G^F$ corresponding to 
$A$ as given in Theorem 1.5. 
Then by Lemma 5.5, we have 
$m_2(R_0) \ne 0$.  
Note that $R_0$ is a linear combination of 
irreducible characters contained in $\CE(G^{F^2}, \{s\})$.  
  Thus by Theorem 5.7, (i), there exists
$s_0 \in \{ s\}$ such that $s_0 \in G^{F^2}$ and that 
$F(s_0) = s_0\iv = s_0$.  
Then $H = Z^0_G(s_0)$ is an $F$-stable reductive subgroup of $G$, 
and so $H$ is split over $\BF_{q^2}$. 
Recall that the almost character $R_0$ is given as in (1.3.1).
Then it is known that $(-1)^{l(z_1)} = \s(H)\s(G)$, where $\s(G)$
is a split rank of $G$ with respect to $\BF_{q^2}$, 
and vice versa for $H$. Since $G, H$ are 
split, we have $l(z_1) = \s(H)\s(G) = 1$.
Let $R_0^H$ be the almost character of $H^{F^2}$ obtained from 
$R_0$ under the correspondence $\r \lra \r\uni$.
Then we see that
\begin{equation*}
\tag{5.9.1}
m_2(R_0) = m_2(R_0^H) = 1.
\end{equation*}
In fact, the first equality follows from Theorem 5.7, (i), 
together with the fact that $l(z_1) = 1$.  The second equality follows from 
Theorem 5.7, (ii), (iii).
\par
Combining Lemma 5.5 with (5.9.1), we have the following theorem.
\begin{theorem}  
Let $G$ be an adjoint simple group of classical type.
Let $A$ be a cuspidal character sheaf, and $\x_A = \x_{A, \f_A}$ 
be the
characteristic function of $A$ on $G^{F^2}$ (defined as in 3.1).  
Let $R_0$ be the almost character of $G^{F^2}$ corresponding to $A$.
Then we have
\begin{equation*}
R_0 = \x_A.
\end{equation*}
\end{theorem}
\par
As a corollary we have the following result, which holds without
any restriction on $p$ nor $q$.
\begin{cor}  
Let $G$ be a connected classical group with connected center.
Then the constants $\z_x$ appearing in Lusztig's, conjecture
(Theorem 1.5) can be determined for $G^{F^2}$ in the following cases;
under the notation of 1.6, assume that $W_{\d} = Z_{\d}$.  Then we have 
\begin{equation*}
R_{x_E} = (-1)^{\dim \vS}\x_{A_E}
\end{equation*}
for any $E \in W_{\d}\wg$.
In other words, we have $\z_{x_E} = (-1)^{\dim \vS}$.
\end{cor}
\begin{proof}
Lemma 1.8 together with the argument in 1.9 shows that 
the determination of $\z_x$ is reduced to the case of 
$\z_0$ (the one corresponding to the cuspidal character sheaf)
in the case of adjoint simple groups.  We know that $\z_0 = 1$ by
Theorem 5.10.  Then the corollary follows from Lemma 1.8, together 
with 1.9 since $L^{F^2}_{w_1} = L^{F^2}$.
\end{proof}
\begin{remark} 
In the case where $\ol G = PSp_{2n}$ or $SO_{2n+1}$, 
Corollary 5.11 gives a complete answer for the determination 
of constants $\z_x$ for $G^{F^2}$ since we have alwyas 
$W_{\d} = Z_{\d}$ in that case.  In the case where 
$\ol G = PSO_{2n}$, the corollary holds if $R_{x_E}$ is a linear 
combination of unipotent characters, i.e., if $A_{x_E} \in \wh G_{\Ql}$. 
But it happens that $W_{\d} \ne Z_{\d}$ for some $\d$. 
\end{remark}

\par\medskip
\section{From $G^{F^2}$ to $G^F$} 
\para{6.1.}
The results Theorem 5.10 and Corollary 5.11 in the previous 
section are only valid for the group $G^{F^2}$.  In this section, 
by using a certain specialization argument, we extend those results
to the group $G^F$ as far as $\x_A$ are concerned with 
unipotent characters.  In the case of $p = 2$, this implies the extension 
of Theorem 5.10 and Corollary 5.11 for the group $G^F$ of split type.
We have the following theorem.  
\begin{theorem}  
Let $G$ be a classical group of split type over $\Fq$.  
Let $A$ be a cuspidal character sheaf contained in $\wh G_{\Ql}$, 
and $\x_A = \x_{A,\f_A}$ be the characteristic function of $A$ over
$G^F$.  Let $R_0$ be the almost character of $G^F$ corresponding to
$A$.  Then we have
\begin{equation*}
R_0 = \x_A.
\end{equation*}
\end{theorem}
As a corollary, we have
\begin{cor}  
Let $G = Sp_{2n}$ or $SO_{2n}$ with $p = 2$.
Assume that $G^F$ is of split type.  Then the constants 
$\z_x$ appearing in Lusztig's conjecture (Theorem 1.5) can 
be determined completely for $Sp_{2n}$, and partly  
for $SO_{2n}$.  More precisely, under the notation of
1.6, assume that $W_{\d} = Z_{\d}$.  Then we have
\begin{equation*}
\tag{6.3.1}
R_{x_E} = (-1)^{\dim \vS}\x_{A_E}
\end{equation*}
for any $E \in W_{\d}\wg$.
In particular, (6.3.1) holds if $A_E \in \wh G_{\Ql}$ 
in the case where $G = SO_{2n}$. 
\end{cor}
\begin{proof}
As in the proof of Corollary 5.11, the determination of $\z_x$
is reduced to that of $\z_0$. 
Assume that $G = Sp_{2n}$ or $SO_{2n}$ with $p = 2$.
In this case, it is known by 3.1, (c), (d), that a unique cuspidal character
sheaf (if it exists) is always contained in $\wh G_{\Ql}$.
Hence Theorem 6.2 can be applied, and the corollary follows from 
Lemma 1.8 (see also Remark 5.12).
\end{proof}
\begin{remark}  
 A similar argument as in Corollary 6.3 works also for the case where
$p \ne 2$.  In particular the formula (6.3.1) holds for the case where 
$A_{E} \in \wh G_{\Ql}$. Thus we 
rediscover the results in Theorem 6.2 in [S2].  Although the argument 
given in [S2] can not be applied to the case where $p = 2$, 
the proof here works simultaneoulsy for arbitrary $p$.
\end{remark}
\para{6.5.}
The remainder of this paper is devoted to the proof of Theorem 6.2.
In what follows, we assume that $G$ is a classical group 
containing a cuspidal unipotent
character.  Hence $G$ is of type $B_n$ or $C_n$ with 
$n = d^2 +d$, or of type $D_n$ with $n = 4d^2$.
We assume further that $G^F$ is of split type. 
We follow the notation in Section 4.
\par
Let $I_q$ be the $G^F$-module $\Ind_{B^F}^{G^F}1$ induced from 
$B^F$ to $G^F$.  Then the irreducible component
of $I_q$ is in bijective correspondence with $W\wg$.  We denote by 
$\r_E$ the irreducible $G^F$-module occurring in $V_q$ corresponding to
$E \in W\wg$.  $\r_E$ gives a unipotent character, which we denote 
by $\r_{\vL_E}$ with $\vL_E \in X_n$. 
Let $\CH$ be the Iwahori-Hecke algebra over $\BQ[u^{1/2}, u^{-1/2}]$
associated to the 
Coxeter system $(W, S)$ with generators $\{ T_s \mid s \in S\}$. 
$\CH$ has a basis $\{ T_w \mid w \in W\}$, where $T_w$ is defined 
as $T_w = T_{s_{i_1}}\dots T_{s_{i_k}}$ for a reduced expression 
$w = s_{i_1}\cdots s_{i_k}$. 
$\CH$ is characterized by the following properties;
\begin{equation*}
\begin{cases}
(T_s -u)(T_s +1) = 0, \\
T_sT_w = T_{sw}  \quad\text{ if } l(sw) = l(w) +1,
\end{cases}
\end{equation*}
where $l: W \to \ZZ_{>0}$ is the length function of $W$.
\par
In the case of type $B_n$ or $D_n$, the generator set $S$ of $W$
is described as follows.  Assume that $W$ is the Weyl group of type
$B_n$.  Then $W$ is realized as a group of signed permutations 
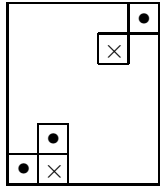
of $I = \{ 1,\bar 1, 2, \bar 2, \dots, n, \bar n\}$.  
The set 
$S$ of generators is given as $S = \{ s_0, s_1, \dots, s_{n-1}\}$
with $s_0 = (1, \bar 1), s_1 = (1,2), \dots, s_{n-1} = (n-1,n)$, 
and we denote by $T_i$ the generator of $\CH$ corresponding to $s_i$.
Note that the subalgebra of $\CH$ generated by $T_1, \dots, T_{n-1}$
is isomorphic to the Iwahori-Hecke algebra of type $A_{n-1}$.
Next assume that $W$ is of tyep $D_n$.  Then $W$ is a subgroup of 
the Weyl group of type $B_n$, generated by 
$s'_0 = (1,\ol 1)(2,\ol 2), s_1, \dots, s_{n-1}$.  We denote by 
$T_0', T_1, \dots, T_{n-1}$ the corresponding generators of $\CH$.
\par
The endomorphism algebra $\End_{G^F}I_q$ is isomorphic to 
the specialized algebra $\Ql\otimes \CH$ via the algebra 
homomorphism $\BQ[u^{1/2},u^{-1/2}] \to \Ql$ by $u \mapsto q$, 
which we denote by $\CH_q$.
We denote by $E_q$ the irreducible representation  of $\CH_q$ 
corresponding to $E \in W\wg$.
Now $I_q$ has a structure of $G^F\times \CH_q$-module, and 
the trace for $g \in G^F, T_w \in \CH_q$ is written as 
\begin{equation*}
\tag{6.5.1}
\Tr((g, T_w), I_q) = \sum_{E \in W\wg}\Tr(g, \r_E)\Tr(T_w, E_q).
\end{equation*}
By replacing $\r_E = \r_{\vL_E}$ by $R_{\vL}$ by using (4.2.3), we have
\begin{equation*}
\tag{6.5.2}
\Tr((g, T_w), I_q) = \sum_{\vL \in X_n}f_{\vL}(w)R_{\vL}(g),
\end{equation*}
where 
\begin{equation*}
\tag{6.5.3}
f_{\vL}(w) = \sum_{E \in W\wg}\{ \vL, \vL_E\}\Tr(T_w, E_q).
\end{equation*}
\par
It is known that  $\Tr(T_w, E_q)$ is a polynomial in $q$ 
in the sense of 2.3.
Hence $f_{\vL}(w)$ is also a polynomial in $q$. 
We are interested in $f_{\vL}(w)$ in the case where 
$\vL$ is the cuspidal symbol $\vL_c$, and we want to find 
some special $w \in W$ such that $f_{\vL_c}(w) \ne 0$.  
Let $W$ be the Weyl group of type $B_n$ or $D_n$.  Then 
any element of $W$ can be expressed as a product of 
positive cycles and negative cycles, where the number 
of negative cycles is even if $W$ is of type $D_n$.
We have the following proposition.
\begin{prop} 
There exists an element $w \in W$ such that 
$f_{\vL_c}(w) \ne 0$, where either $w$ is a Coxeter element 
in $W$, or $w$ contains a positive cycle of length $\ge 2$. 
\end{prop}
\para{6.7.} 
The proof of the proposition will be given in Section 7.  
Here assuming the proposition, we continue the proof of the 
theorem.  We prove the theorem by induction on the semisimple 
rank of $G$, and so we assume that the theorem holds for the classical 
groups of the smaller semisimple rank.  
Let $A$ be the cuspidal character sheaf on $G$ as in the theorem, 
and let $C$ be the conjugacy class which is the support of $A$.
We choose $g = su = us \in C^F$ as in 3.1.  We choose $w \in W$ as in 
Proposition 6.6. 
We consider the equation (6.5.2) simultaneously for the groups 
$G^{F^m}$  for any integer $m \ge 1$.
Note that $g \in C^{F^m}$ and it has a uniform description for 
any $m \ge 1$ since the split unipotent element for $G^F$ 
is split for any extended group  $G^{F^m}$.  
Then we can write (6.5.2) as 
\begin{equation*}
\tag{6.7.1}
\Tr((g, T_w), I_{q^m}) = \sum_{\vL \ne \vL_c}f_{\vL}(w)(q^m)
                                   R^{(m)}_{\vL}(g) 
                        + f_{\vL_c}(w)(q^m)R^{(m)}_{\vL_c}(g),
\end{equation*}
where $T_w$ is an element of $\CH_{q^m}$, and 
$R^{(m)}_{\vL}$ denotes the almost character of $G^{F^m}$.
By induction hypothesis and by Remark 6.4, the formula (6.3.1) 
holds for $R^{(m)}_{\vL}$ if $\vL \ne \vL_c$.   
We show the following lemma.
\begin{lemma}  
Assume that $\vL \ne \vL_c$.  Then under the induction hypothesis, 
$R^{(m)}_{\vL}(g)$ is a rational function in $q$.
\end{lemma}
\begin{proof}
Let $A_{\vL}$ the character sheaf corresponding to 
$R_{\vL}$, and denote by $\x^{(m)}_{A_{\vL}}$ the characteristic
function on $G^{F^m}$ associated to $A_{\vL}$.
In view of (6.3.1), it is enough to show that 
$\x^{(m)}_{A_{\vL}}(g)$ ia a rational function in $q$.
Since $A \in \wh G_{\Ql}$, by applying Corollary 4.10 in [S2]
($a_0 = 1$ in the notation there since $F$ is of split type),
the computation of $\x_{A_{\vL}}^{(m)}(g)$ is reduced to that of 
$\x_{A_{\vL}}^{(m)}(u)$ for a subgroup $H^{F^m}$.
So we consider $\x^{(m)}_{A_{\vL}}$ on $G^{F^m}\uni$.
$A_{\vL}$ is a direct summand of a certain complex 
$K = \ind_P^GA_0$, where $A_0 \in \wh L_{\Ql}$ is a cuspidal character sheaf
of a Levi subgroup $L$.  
Then the computation of $\x^{(m)}_{A_{\vL}}|_{G^{F^m}\uni}$ 
is reduced to the
computation of the generalized Green functions of $G^{F^m}$
associated to $L^{F^m}$.
Hence by Corollary 2.4, we obtain the lemma.
\end{proof}
\para{6.9.}
Next we consider the left hand side of (6.7.1).
The $G^F$-module $I_q$ is a permutation representation of 
$G^F$ on $G^F/B^F$.  Since the action of $G^F$ on $G^F/B^F$ is 
independent of the isogeny, we may assume that $G = Sp_N$ or 
$SO_N$ and that $g \in G^F$.    
Let $V$ be the vector space over $\BF_q$ of 
$\dim V = N$, equipped with a non-degenerated alternating form 
(resp. symmetric bilinear form) $f$ on $V$ if $G = Sp_N$ 
(resp. $SO_N$).  In the case where $G = SO_N$ with 
$p = 2$ and $N$ is even, 
we also consider the quadratic form $Q$ on $V$. 
Now the set $G^F/B^F$ may be 
identified with the set of flags $\CF_q$ as follows;
A flag $\BF = (V_0 \subset V_1 \subset \cdots \subset V_n)$
is a sequence of subspaces of $V$ such that $\dim V_i = i$
and that $V_i$ are isotropic with respect to $f$, where 
$N = 2n, 2n+1$ or $2n$ according to the cases where 
$G = Sp_{2n}, SO_{2n+1}$ or $SO_{2n}$.  In the case where
$p = 2$ and $G = SO_{2n}$, we assume further that the restriction 
of $Q$ on $V_n$ is zero. 
Now in the case where $G = Sp_{2n}$ or $SO_{2n+1}$, 
$\CF_q$ consists of all flags on $V$.
$G^F$ acts naturally on $\CF_q$ via 
$x : (V_0 \subset \cdots \subset V_n) \mapsto 
     (xV_0 \subset \cdots \subset xV_n)$ for $x \in G^F$, 
and the $G^F$-set $\CF_q$ is identified with the 
$G^F$-set $G^F/B^F$.
In the case where $G = SO_{2n}$, 
we define $\wt \CF_q$ as the set of all flags 
on $V$ as above, 
then $G^F$ acts on $\wt\CF_q$ withtwo $G^F$-orbits, 
$\CF_q$ and $\CF'_q$.  
Either of them can be identified with $G^F/B^F$, and 
we have a natural bijection between 
$\CF_q$ and $\CF'_q$, which is given in the form
$(V_0 \subset V_1 \subset \cdots \subset V_n) \mapsto 
 (V_0 \subset V_1\subset \cdots \subset V_n')$,
(only the term $V_n$ is changed to $V_n'$).
\par
We consider the vector space $\CJ_q$ over $\Ql$ with 
basis $\CF_q$, which is identified with $I_q$.
By the identification $I_q \simeq \CJ_q$, $\CH_q$ acts 
on $\CJ_q$, whose action is given as follows; 
let $\BF = (V_0 \subset \cdots \subset V_n)$.
For $i = 1, \dots, n-1$, we have
\begin{equation*}
\BF T_{n-i} = \sum_{W \ne V_i}
           (V_0 \subset \cdots \subset V_{i-1} \subset W 
               \subset V_{i+1} \subset \cdots \subset V_n), 
\end{equation*}
where the sum is taken over all the isotopic subspaces $W$ such that 
$V _{i-1} \subset W \subset V_{i+1}$ and that $\dim W = i$, $W \ne V_i$.
In the case of $B_n$, we have
\begin{equation*}
\BF T_0 = \sum_{W \supset V_{n-1}}
           (V_0 \subset \cdots \subset V_{n-1} \subset W), 
\end{equation*}
where the sum is taken over all the isotropic subspaces $W$
such that $\dim W = n$ and $W \ne V_n$.
In the case of type $D_n$, we have
\begin{equation*}
\BF T'_0 = \sum_{W \supset W' \supset V_{n-2}}
           (V_0 \subset \cdots \subset V_{n-2} \subset W' \subset W), 
\end{equation*}
where the sum is taken over the isotropic subspaces 
$W \supset W'$ such that $\dim W = n, \dim W' = n-1$ and
$W'$ contains $V_{n-2}$ and some more conditions.
\par
It follows from the description of the action of $\CH$ on $\CI_q$, 
we see that 
\par\medskip\noindent
(6.9.1) \ Assume that $w \in W$ contains a positive cycle of 
length $\ge 2$.  Then 
there exists $k \ge 1$ such that for any 
$\BF = (V_0 \subset \cdots \subset V_n)$, $\BF T_w$ is  
a linear combination of $\BF' = (V_0' \subset \cdots \subset V_n')$
such that $V_k = V_k'$.
\par\medskip
\par
We show the following lemma.
\begin{lemma} 
Under the induction hypothesis, $\Tr((g, T_w), I_{q^m})$
is a rational function  in $q$.
\end{lemma}
\begin{proof}
The following argument was inspired by [HR2], where the combinatorial 
properties of $\Tr ((u, T_w), I_q)$ is discussed in the case of 
$GL_n(\Fq)$ with a unipotent element $u$.
First assume that $w$ contains a positive cycle of length $\ge 2$. 
Then by (6.9.1), for any flag 
$\BF = (V_0 \subset \cdots \subset V_n)$, $\BF T_w$ is a linear
combination of $\BF' = (V_0' \subset \cdots \subset V_n')$ 
such that $V_k = V_k'$. 
We now prepare a notation.  If $v \in \CJ_q$, and $\BF \in \CF_q$, 
we denote by $v|_{\BF}$ the coefficient of $\BF$ in the expression 
of $v$ as a linear combination of base vectors. 
Let $\BF = (V_0 \subset \cdots \subset V_n) \in \CF_q$, 
and assume that $g\BF T_w|_{\BF} \ne 0$.
Since $g\BF T_w|_{\BF} = \BF T_w|_{g\iv \BF}$,
$g\iv\BF = (V_0' \subset \cdots \subset V_n')$ is of the 
form that $V_k = V_k' = g\iv V_k$.
It follows that $V_k$ is stabilized by $g$.  Thus we have
\begin{equation*}
\begin{aligned}
\Tr( (g,T_w), I_q) &= \sum_{\BF \in \CF_q}
       g\BF T_w|_{\BF}  \\
 &= \sum_{W}\sum_{\BF = (V_0 \subset \cdots \subset W 
         \subset \cdots \subset V_n)}
                   g\BF T_w|_{\BF} \\
 &= \sum_{W}\sum_{\BF' = (V_0 \subset \cdots \subset W)}g\BF' T_w|_{\BF'}
    \sum_{\BF'' = (W \subset \cdots )}g\BF''T_w|_{\BF''}
\end{aligned}
\end{equation*}
where $W$ runs over all the isotropic subspaces in $V$ such that
$\dim W = k$ and that $gW = W$.
Let $H' = GL(W)$,  and $H''$ be the group of isometries $Sp(\ol W)$
or $SO(\ol W)$ for $\ol W = W\per/W$.  Let $I^W_q \simeq \CJ^W_q$ 
be the corresponding induced modules for ${H'}^F$, and similarly 
define $I^{\ol W}_q \simeq \CJ^{\ol W}_q$ for ${H''}^F$.  
$g$ acts naturally on $W$ (resp. on $\ol W$), and we denote by 
$g_W \in {H'}^F$ (resp. $g_{\ol W} \in {H''}^F$) the corresponding elements.
Also the action of $T_w$ on $\CI_q$ induces an action on $\CJ_q^{W}$
(resp. on $\CJ_q^{\ol W}$) which is given by $T_{w'}$ (resp. $T_{w''}$)
with an element $w'$ (resp. $w''$) in the Weyl group of $H'$ (resp. $H''$).
Then the last sum can be written as 
\begin{equation*}
\tag{6.10.1}
\Tr((g,T_w), I_q) = 
   \sum_{W}\Tr((g_W, T_{w'}), I_q^{W})\Tr((g_{\ol W}, T_{w''}), I_q^{\ol W}).
\end{equation*}
\par
Let $\BF_0 = (V^0_0 \subset\cdots \subset V^0_n)$ be the standard 
flag whose stabilizer in $G^F$ is $B^F$, and put $W_k = V_k^0$.
Then there exists an $F$-stable  maximal parabolic subgroup 
$P$ of $G$ containing $B$ such that $P^F$ is the stabilizer of 
$W_k$ in $G^F$.  Let $L$ be an $F$-stable 
Levi subgroup of $P$ containing $T$. Then $L$ is isomorphic to 
$L' \times L''$, where $L' = GL_k$ and $L''$ is 
a similar group as $G$ of rank $n-k$.
Let $g'_1, \dots, g'_r$ (resp. $g_1'', \dots, g_s''$) 
be representatives of the conjugacy classes 
in ${L'}^F$ (resp. ${L''}^F$) such that $g_{ij} = (g'_i, g_j'') \in L^F$ 
is conjugate to $g$ under $G^F$.
We denote by $X_q^{ij}$ the set of $W$ such that $W = x W_k$ and that 
$x\iv gx$ is conjugate to $g_{ij}$ in $L^F$. 
Then (6.10.1) implies that 
\begin{equation*}
\tag{6.10.2}
\Tr((g,T_w), I_q)  
     = \sum_{i=1}^r\sum_{j=1}^s|X_q^{ij}|
         \Tr((g'_i,T_{w'}), I_q^{L'})\Tr((g_j'', T_{w''}), I_q^{L''}),
\end{equation*}
where $I_q^{L'}, I_q^{L''}$ are corresponding induced modules 
for ${L'}^F, {L''}^F$.
Choose $x_{ij} \in G^F$ such that 
$g = x_{ij}g_{ij}x_{ij}\iv$ for $i, j$.
Then $X_q^{ij}$ is in bijection with the set $Z_G(g)^Fx_{ij}L^F/L^F$, hence 
$|X_q^{ij}| = |Z_G(g)^F|/|Z_L(g_{ij})^F|$.
\par
We now consider (6.10.2) for any $G^{F^m}$.  
Then the choice of representatives $g_{ij} \in G^{F^m}$
does not depend on $m$, and we see that $|X_{q^m}^{ij}|$ is a rational
function in $q$. On the other hand, one can write as in (6.5.1)
\begin{equation*}
\Tr((g''_j, T_{w''}), I_q^{L''}) = \sum_{E \in {W_{L''}}\wg}
                           \Tr(g''_j, \r_E)\Tr(T_{w''}, E_q),
\end{equation*}
where $W_{L''}$ is the Weyl group of $L''$.
By induction hypothesis, $R^{(m)}_{\vL}(g''_j)$ is a rational 
function in $q$ for any almost character $R_{\vL}$ of ${L''}^F$.
Hence $\Tr(g''_j,\r_E^{(m)})$ is a rational function in $q$.  
It follows that $\Tr((g''_j, T_{w''}), I_{q^m}^{L''})$ is a 
rational function in $q$.
Similarly, and as it is known since $L' = GL_k$, 
$\Tr((g'_i, T_{w'}), I_{q^m}^{L'})$ is a rational function 
in $q$.
Thus we conclude that $\Tr((g, T_w), I_{q^m})$ is a rational 
function in $q$ as asserted.
\par
Next assume that $w$ is a Coxeter elment of $W$.
We note that 
\begin{equation*}
\tag{6.10.3}
\Tr((x, T_w), I_q) = \begin{cases}
              q^r &\quad\text{ if $x_u$ is regular unipotent,}  \\
              0   &\quad\text{ otherwise,}
                      \end{cases}
\end{equation*}
for $x \in G^F$, where $x_u$ is the unipotent part of 
$x$ and $r$ is the semisimple rank of $G$. 
In fact (6.10.3) is discussed in [HR2, Prop. 3.2] 
in the case where $G = GL_n$.  The argument there 
works in general if we notice that 
$Z_G(v) = Z_U(v)$ for a regular unipotent element 
$v \in U^F$ and that $|Z^F_U(v)| = q^r$, where $U$ is the
unipotent radical of $B$.
(6.10.3) implies that $\Tr((g, T_w), I_{q^m})$ is 
a polynomial in $q$. 
Hence the lemma holds. 
\end{proof}
\para{6.11}
We now prove the theorem. In the formula (6.7.1)
the lef hand side is a rational function in $q$ by 
Lemma 6.10.  For for $\vL \ne \vL_c$, 
$R^{(m)}_{\vL}(g)$ is a rational function in $q$ by 
Lemma 6.8.  Since $f_{\vL}(w)$ is a polynomial in $q$, and 
$f_{\vL_c}(w) \ne 0$ by Proposition 6.6, we see that 
$R^{(m)}_{\vL_c}(g)$ is a rational function in $q$.
By Theorem 1.5, one can write 
$R^{(m)}_{\vL_c}(g) = \z^{(m)}\x^{(m)}_{A_0}(g)$ with 
some $\z^{(m)} \in \Ql^*$ of absolute value 1, 
where $A_0 = A_{\vL_c}$ is the 
cuspidal character sheaf.  We know that $\x^{(m)}_{A_0}(g)$ 
is a non-zero polynomial in $q$. 
We also know by Theorem 5.10 that $\z^{(m)} = 1$ for even $m$.
It follows that $R^{(m)}_{\vL_c}(q)/\x^{(m)}_{A_0}(g)$ is 
a rational function in $q$, and takes the value 1 for any power 
of $q^2$.  Hence $R^{(m)}_{\vL_c}(g) = \x^{(m)}_{A_0}(g)$ for any $m$, 
and we have $\z^{(m)} = 1$. 
This shows that $R_{\vL_c} = \x_{A_0}$, and the theorem is proved
(modulo Proposition 6.6).
\par\bigskip  
\section{Proof of Proposition 6.6}
\para{7.1.}
Recall that $\F^1_n$ is the set of symbols of rank $n$ and 
defect 1 as in 4.1. 
In the case where $W$ is the Weyl group of type $B_n$, 
the set $W\wg$ is in bijection with $\F^1_n$.  The 
correspondence is given as follows; let $\CP_n$  be the set of 
of double partitions $(\la, \mu)$ such
that $|\la| + |\mu| = n$.   Then $W\wg$ is parametrized by 
$\CP_n$.  For a double partition 
$(\la, \mu) \in \CP_n$, we write 
$\la = (\la_0 \le \la_1 \le  \cdots \le \la_m)$ and 
$\mu = (\mu_1 \le \mu_2 \le \cdots \le \mu_m)$ with $\la_i, \mu_i \ge 0$ for
some integer $m$.  We put $a_i = \la_i + i$, 
$b_i = \mu_i + (i-1)$, and define the sets $S, T$ by 
$S = \{ a_0, a_1, \dots, a_m\}$ and 
$T = \{ b_1, \dots, b_m\}$.  Then $\vL = \binom{S}{T} \in \F_n^1$, 
and this gives a bijective correspondence between 
$\F^1_n$ and $\CP_n$,
and so gives a bijection between $\F^1_n$ and $W\wg$.
As in 6.5, we denote by $\vL_E$ the symbol in $\F^1_n$ corresponding to 
$E \in W\wg$.
Assume that $n = d^2 + d$, and let $\vL_0$ be the special symbol 
in $\CF_c$.  Then $\vL_0 = \vL_E$, where $E$ corresponds to 
$(\la, \mu) \in \CP_n$ such that 
$\la = (0 \le 1 \le 2 \le \cdots \le  d )$, 
$\mu = (1 \le 2 \le \cdots \le d)$.
\par
Recall that $\wt \F_n^0$ is the set of symbols of rank $n$ 
and defect 0 as in 4.1.  In the case where $W$ is the Weyl group
of type $D_n$, $W\wg$ is in bijection with $\wt\F_n^0$, which 
is given as follows; let $\wt \CP_n$ be the set of unordered 
partitions $(\la, \mu)$ such that $|\la| + |\mu| = n$, where 
$(\la, \la)$ is counted twice.  Then $W\wg$ is parametrized 
by $\wt\CP_n$.  For $(\la,\mu) \in \wt\CP_n$,  we write
$\la = (\la_1 \le \cdots \le \la_m)$, 
$\mu = (\mu_1 \le \cdots \le \mu_m)$ with $\la_i,\mu_j \ge 0$ 
for some integer $m \ge 1$.  We put $a_i = \la_i + (i-1)$, 
$b_i = \mu_i + (i-1)$, and define the sets $S, T$ by 
$S = \{ a_1, \dots, a_m\}$, $T = \{ b_1, \dots, b_m \}$.
Then $\vL = \binom {S}{T} \in \wt\F^0_n$ and this gives a
bijective correspondence between $\wt\F^0_n$ and $W\wg$.
As in 6.5, we denote by $\vL_E$ the symbol in $\wt\F^0_n$ corresponding 
to $E \in W\wg$.
Assume that $n = 4d^2$ and let $\vL_0$ be the special symbol
in $\CF_c$. 
Then $\vL_0 = \vL_E$, where $E$ corresponds to 
$(\la,\mu) \in \wt\CP_n$ such that 
$\la = (0 \le 1 \le \cdots \le 2d-1)$, 
$\mu = (1 \le 2 \le \cdots \le 2d)$
\par
First we show the following lemma.
\begin{lemma} 
Assume that $E \in W\wg$ is such that $\vL_E \in \CF_c$.  
If $E$ corresponds to $(\la, \mu) \in \CP_n$ 
(resp. $(\la,\mu) \in \wt\CP_n$), in the case where $W$ is of 
type $B_n$ (resp. $D_n$),  then 
we have 
\begin{equation*}
\{ \vL_c, \vL_E\} = \begin{cases}
               \displaystyle\frac{1}{2^d}(-1)^{|\la| + d(d+1)/2} 
                   \quad\text{ if  $G$ is of type $B_n$}, \\ \\
               \displaystyle\frac{1}{2^{2d-1}}(-1)^{|\la| + d(2d-1)}
                    \quad\text{ if $G$ is of type $D_n$}.
                    \end{cases}
\end{equation*}
\end{lemma}
\begin{proof}
In the case where $\CF = \CF_c$ is the cuspidal family,
$Z_1 = \{ 0,1, \dots, 2d\}$ (resp. $Z_1 = \{ 0, 1, \dots, 4d-1\}$) 
and $M_0 = \{ 1,3, \dots, 2d-1 \}$ 
(resp. $M_0 = \{ 1,3, \dots, 4d-1\}$) if $G$ is of type $B_n$ 
(resp. $D_n$) by 4.3.
Moreover $\vL_c = \vL_M$ with $M = \emptyset$.  
Then $M^{\sharp} = M_0$, and for any $M' \subset Z_1$ such that 
$|M'| = d$ (resp. $|M'| = 2d$), 
we have ${M'}^{\sharp} \cap M^{\sharp} = M_0 - M'$.
Now $M'$ is written as $M' = M'_{\odd} \coprod M'_{\ev}$, where 
$M'_{\odd}$ (resp. $M'_{\ev}$) is the subset of $M'$ consisting 
of odd numbers (resp. even numbers).  
We have $M' \cap M_0 = M'_{\odd}$ and so 
$M_0 - M' = M_0 - M'_{\odd}$.
Moreover we have  
$Z_1 - M' = ((Z_1)_{\ev} - M'_{\ev}) \coprod 
              (M_0 - M'_{\odd})$, where
$(Z_1)_{\ev}$ is defined similarly.
Assume that $M'$ corresponds to $(\la, \mu) \in \CP_{n}$
(resp. $(\la,\mu) \in \wt\CP_n$).
Let $\g = \sum_{i=1}^d i = d(d+1)/2$ 
(resp. $\g = \sum_{i=1}^{2d-1}i = d(2d-1)$).  Then  
\begin{equation*}
\begin{aligned}
|\la| &= (\sum_{a \in Z_1-M'}a) - \g   \\
      &\equiv  (\sum_{a \in M_0 - M'_{\odd}}a) - \g  
                                \pmod 2  \\
      &\equiv |M_0 - M'_{\odd}| - \g
                                \pmod 2.
\end{aligned}
\end{equation*}

This proves the lemma.
\end{proof}
\para{7.3.}
We consider $f_{\vL_c}(w) \in \ZZ[u^{1/2}, u^{-1/2}]$ 
for some $w \in W$, and compute it by 
making use of the Murnaghan-Nakayama formula for $\CH$ due to 
Halverson and Ram [HR1]. We use the notation in 6.5.
\par
First assume that $W$ is of type $B_n$ and $\CH$ is the corresponding 
Hecke algebra. 
For each $1 \le k \le n$, we define $L_k \in \CH$ inductively as
$L_1 = T_0$ and $L_k = T_{k-1}L_{k-1}T_{k-1}$ for $k = 2, \dots, n$. 
For $1 \le k < l \le n$, we define 
\begin{equation*}
R_{kl} = T_{k}T_{k+1}\cdots T_{l-1}, \qquad 
R_{k\bar l} = L_kT_{k}T_{k+1}\cdots T_{l-1}.
\end{equation*}
For an element $i$ or $\bar i$ in $I$, we put 
$|i| = |\bar i| = i$.
For a sequence $\Br = (r_1, \dots, r_k)$ of elements of $I$
such that $|r_1| < |r_2| < \cdots < |r_k|$, 
we define $T_{\Br} \in \CH$ by 
\begin{equation*}
T_{\Br} = R_{1,r_1}R_{|r_1|+1,r_2}\cdots R_{|r_k-1|+1, r_k}.
\end{equation*}
Then $T_{\Br}$ coincides with $T_{w_{\Br}}$, where 
$w_{\Br} \in W$ is given by a cyclic notation of the signed 
permutation, 
\begin{equation*}
\tag{7.3.1}
\begin{split}
w_{\Br} = &(1,2, \dots, |r_1|-1, r_1)
      (|r_1|+1, |r_1|+2, \dots, |r_2|-1, r_2)  \\
  &\cdots(|r_{k-1}|+1, \dots, |r_k|-1, r_k).
\end{split}
\end{equation*}
We also use the following cycle type expression of $w_{\Br}$ 
\begin{equation*}
\tag{7.3.2}
w_{\Br} = [l_1, \dots, l_r],
\end{equation*}
where $l_i \in I$ is such that $|l_i| = |r_i| - |r_{i-1}|$
and $l_i$ is barred if $r_i$ is barred. 
For example, if $\Br = (\bar 1, \bar 4, 7, 12)$, then
$w_{\Br} = [\bar 1, \bar 3, 3, 5]$. 
\par
We now prepare some notation related to the skew diagram.
Let $\la$ be a double partition of size $n$.
Apart from the notation in 7.1, we express it as 
$\la = (\la^{\a}, \la^{\b})$, where $\la^{\a}, \la^{\b}$ are
partitions.  For $\mu \in \CP_n$, we write $\mu \subseteq \la$
if $\mu^{\a} \subseteq \la^{\a}$ and $\mu^{\b} \subseteq \la^{\b}$.
The Young diagram of $\la$ is defined as a pair of Young diagrams
of $\la^{\a}$ and $\la^{\b}$.  We often identify the double 
partition and the corresponding Young diagram.  For double partitions 
$\mu \subseteq \la$, the skew diagram 
$\la/\mu = ((\la/\mu)^{\a}, (\la/\mu)^{\b})$ is defined naturally.
For each node $x$ in the skew diagram $\la/\mu$, the content $ct(x)$ 
is defined as follows;
\begin{equation*}
\tag{7.3.3}
ct(x) = \begin{cases}
           u^{j-i+1} \quad\text{ if $x$ is in position $(i,j)$ in 
                          $(\la/\mu)^{\a}$}, \\
           -u^{j-i}  \quad\text{ if $x$ is in position $(i,j)$ in 
                          $(\la/\mu)^{\b}$}.
         \end{cases}
\end{equation*}
\par
The skew diagram $X$ is called a border strip if it is connected
and does not contain any 2$\times 2$ block of nodes  (``connected''
means that two nodes are connected horizontally or vertically). 
The skew diagram $X$ is called a broken border strip if its 
connected components are border strip. 
Note that a double partition $(\a, \b)$ with both $\a, \b$ non-empty
consists of two connected components.
For a border strip $X$, a sharp corner is a node with no
node above it and no node to its left.  A dull corner in a border
strip is a node which has a node to its left and a node above it
(and so has no node directly northwest of it). 
\par
For a skew diagram $X$, 
let $\ZC$ be the set of connected components of $X$, and 
put $m = |\ZC|$, the number of connected components of $X$.
We define $\vD(X), \ol\vD(X) \in \ZZ[u^{1/2},u^{-1/2}]$ as follows;
\begin{equation*}
\begin{aligned}
\vD(X) &= \begin{cases}
                 \displaystyle
                 (u^{1/2} - u^{-/2})^{m-1}\prod_{Y \in \ZC}
                      (u^{1/2})^{c(Y) -1}(-u^{-1/2})^{r(Y)-1} \\
                  \hspace{4cm}
          \text{ if $X$ is a broken border strip,}\\
                 0 \quad\text{ otherwise}.
               \end{cases}  
                     \\  \\
\ol\vD(X) &= \begin{cases}
                     \displaystyle
                    (u^{1/2})^{c(X)-1}(-u^{-1/2})^{r(X)-1}
                       \prod_{y \in DC}ct(y)\iv
                       \prod_{z \in SC}ct(z) \\
                   \hspace{4cm}
                 \text{ if $X$ is a (connected) border strip,} \\
                 0 \quad\text{ otherwise,}
                   \end{cases}
\end{aligned}
\end{equation*}
where $SC$ and $DC$ denote the set of sharp corners and dull corners 
in a border strip, and $r(X)$ (resp. $c(X)$) is the number
of rows (resp. columns) in the border strip $X$. 
\par
The Murnaghan-Nakayam formula for $\CH$ by Halverson-Ram 
is given as follows.  Note that in the formula below, 
$l'(w)$ denotes the number of $s_1, \dots, s_{n-1}$ 
(excluding $s_0$) occurring in the reduced expression 
of $w \in W$
\begin{theorem} [{[HR1, Theorem 2.20]}]  
Assume that $W$ is of type $B_n$.  
Let $E_u^{\la}$ be the irreducible representation of
$\CH$  associated to $\la \in \CP_n$.  Then 
\begin{equation*}
\Tr(T_{w_{\Br}}, E_u^{\la}) = u^{l'(w_{\Br})/2}
              \sum_{\emptyset = \mu^{(0)}  \subseteq
     \mu^{(1)} \subseteq \cdots\subseteq \mu^{(k)} = \la}
      \vD(\mu^{(1)})\vD(\mu^{(2)}/\mu^{(1)})\cdots 
         \vD(\mu^{(k)}/\mu^{(k-1)}),
\end{equation*}
where the sum is taken over all the sequences 
$\emptyset = \mu^{(0)}  \subseteq
     \mu^{(1)} \subseteq \cdots\subseteq \mu^{(k)} = \la$
such that $|\mu^{(k)}/\mu^{(k-1)}| = |r_k| - |r_{k-1}|$ and the factor
$\vD(\mu^{(k)}/\mu^{(k-1)})$ is barred if $r_k$ in $\Br$ is
barred.  
\end{theorem}
\para{7.5.}
Next assume that $W$ is the Weyl group of type $D_n$ and $\CH$ is the
corresponding Hecke algebra.  Then under the notation of
6.5, we define $L'_k$ by $L'_1 = 1, L'_2 = T'_0 T_1$, and
$L'_k = T_{k-1}L'_{k-1} T_{k-1}$ for $k = 3, \dots, n$. 
For $1 \le k < l \le n$, we define 
\begin{equation*}
R_{kl} = T_kT_{k+1}\cdots T_{l-1}, \quad 
R_{k\bar l} = L_k'T_kT_{k+1}\cdots T_{l-1}.
\end{equation*}
For a sequence $\Br = (r_1, \dots, r_k)$ of elements $I$ 
such that $|r_1| < |r_2| < \cdots < |r_k|$ and that the even 
numbers of $r_i$ are barred, we define 
$T_{\Br} \in \CH$ by 
\begin{equation*}
T_{\Br} = R_{1,r_1}R_{|r_1|+1, r_2}\cdots R_{|r_{k-1}|+1, r_k}.
\end{equation*}
Note that $T_{\Br}$ does not always correspond to $T_w$ for some 
$w$, but it corresponds to $T_{w_{\Br}}$ for $w_{\Br} \in W$ as given 
in (7.3.1) in the following two cases, 
\par
(i) $r_i > 0$ for $i = 1, \dots, k$, 
\par
(ii) $r_1 = -1, r_2 < 0$ and $r_i > 0$ for $i = 3, \dots, k$. 
\par
In what follows, we only consider $\Br$ as above, and so assume 
that $T_{\Br} = T_{w_{\Br}}$.  Note that $w_{\Br}$ is regarded as 
an element of the Weyl group of type $B_n$ in the notation of 7.3.
\par
As in 7.3, we consider the skew diagram $\la/\mu$, and define 
$ct'(x)$ by modifying (7.3.2),  
\begin{equation*}
ct'(x) = \begin{cases}
           u^{j-i} \quad\text{ if $x$ is in position $(i,j)$ in 
                          $(\la/\mu)^{\a}$}, \\
           -u^{j-i}  \quad\text{ if $x$ is in position $(i,j)$ in 
                          $(\la/\mu)^{\b}$}.
         \end{cases}
\end{equation*}
For a skew diagram $X$, we define 
$\vD(X), \ol\vD'(X) \in \ZZ[u^{1/2}, u^{-1/2}]$ by the formula 
in 7.3, but for $\ol\vD'(X)$, we modify the definition of $\ol\vD(X)$
by replacing $ct(x)$ by $ct'(x)$.  
Concerning the irreducible characters of $\CH$, we have the following 
result.
\begin{theorem}[{[HR1, Theorem 4.21]}]  
Assume that $W$ is of type $D_n$.  Let $E_u^{\la}$ be the 
irreducible representation of $\CH$ associated to $\la \in \wt\CP_n$
such that $\la^{\a} \ne \la^{\b}$.
Then $\Tr(T_{w_{\Br}}, E_u^{\la})$ can be computed by the formula
in Theorem 7.4 for type $B_n$, by replacing $\ol\vD(X)$ by 
$\ol\vD'(X)$.
\end{theorem}
\para{7.7.} Let  $\CF_c$ is the cuspidal family as in 4.3, 
where $\CF_c$ is a subset of $\F_n$ or $\F_n^+$.
Let $\CF^1_c$ (resp. $\CF_c^0$) be the set of symbols of 
defect 1 (resp. defect 0) contained 
in $\CF_c$.  First assume that $\CF_c \subset \F_n$.  
Let $\CP_n^d$ be the set of double partitions 
$\la$ such that $\vL_{E^{\la}} \in \CF_c$.  Then $\la \in \CP_n^d$
can be written as $\la = (\a, \b)$ with 
$\a : \a_1 \ge \a_2 \ge \dots \ge \a_d \ge \a_{d+1} \ge 0$, and 
$\b : \b_1 \ge \b_2 \ge \dots \ge \b_d \ge 0$.  
Let $\b^* : \b_1^* \ge \b_2^* \ge \cdots $ be the dual partition of $\b$.   
The following fact is easily checked.
\begin{equation*}
\tag{7.7.1}
\CF_c^1 \simeq \CP_n^d = \{ \la = (\a, \b) \in \CP_n \mid 
  \a_i + \b^*_{d - i +2} = d \text{ for } 1 \le i \le d+1\}.
\end{equation*}
\par 
Next assume that $\CF_c \subset \F_n^+$.  
Let $\wt\CP_n^d$ be the set of double partitions $\la \in \wt \CP_n$ 
such that $\vL_{E^{\la}} \in \CF_c$, where 
$\la = (\a, \b)$ with $\a: \a_1 \ge \cdots \ge \a_d \ge 0$, 
$\b : \b_1 \ge \cdots \ge \b_d \ge 0$
(in this case always $\a \ne \b$).
Let $\b^*$ be the dual partition of $\b$.  Then we have
\begin{equation*}
\tag{7.7.2}
\CF_c^0 \simeq \wt\CP_n^d = 
  \{ \la = (\a, \b) \in \wt\CP_n \mid 
                           \a_i + \b^*_{2d-i+1} = 2d 
        \text{ for $1 \le i\le d$}\}.
\end{equation*}
\par
(7.7.1) shows that $\la = (\a, \b) \in \CP_{n,2}^d$ is 
obtained from a diagram $\g = (d^{d+1})$ of rectangular 
shape as follows; take any partition $\a \subseteq \g$, 
and let $\b$ be the dual of the partition obtained by rearranging 
the skew diagram $\g/\a$. 
Similarly, (7.7.2) shows that $\la = (\a, \b) \in \wt\CP_n^d$ is
obtained from the diagram $\g = ((2d)^{2d})$ of rectangular shape 
by a similar process as above.
\par
For example, in the case of type $B_n$ with $d = 2$, 
we have $n = d(d+1) = 6$ and 
\begin{equation*}
\begin{split}
\CP_{6}^2 = \{ &(21;21), (2^2;1^2), (1^2;31), (2^3;-), (2^21;1), \\
                &(21^2;2), (1^3;3), (2;2^2), (1;32), (-;3^2)\}.
\end{split}
\end{equation*}
In the case of type $D_n$ with $d = 1$, we have $n = 4d^2 = 4$ and 
\begin{equation*}
\wt\CP_4^1 = \{ (2^2; -), (21^2; 1), (2, 1^2) \}.
\end{equation*}

\para{7.8.}
Let $w_{\Br}$ be an element in $W$ associated to some 
$\Br = (r_1, \dots, r_k)$ as in (7.3.1).
By Lemma 7.2 and by (6.5.3), we have
\begin{equation*}
\tag{7.8.1}
f_{\vL_c}(w_{\Br}) = \d 
       \sum_{(\a,\b)}(-1)^{|\a|}
                 \Tr(T_{w_{\Br}}, E_u^{(\a,\b)}), 
\end{equation*}
where  the sum is taken over all $\la = (\a, \b)$ in $\CP_n^d$
 (resp. in $\wt \CP_n^d$), and the constant $\d$ is given as  
$\d = (-1)^{d(d+1)/2}2^{-d}$ (resp. $(-1)^{d(2d-1)}2^{-2d+1}$)
if $G$ is of type $B_n$ (resp. $D_n$).
\par
We shall compute this sum for some specific choice of $\Br$ 
by applying the Murnaghan-Nakayama formula (Theorem 7.4 or 
Theorem 7.6).
In order to discuss the case $B_n$ and $D_n$ simultaneously, 
we consider the following setting.  Let 
$\CP_n^{a,b}$ be the set of double partitions 
$\la = (\a, \b) \in \CP_n$, where 
$\a: \a_1 \ge \cdots \ge \a_a \ge 0, \b: \b_1 \ge \cdots \ge \b_b \ge 0$ and
$n = ab$, such that   
\begin{equation*}
\tag{7.8.2}
\CP_n^{a,b} = \{ \la = (\a, \b) \in \CP_n 
   \mid \a_i + \b^*_{a-i+1} = b \text{ for $1 \le i \le a$}\}.
\end{equation*}
Hence $\CP_n^{a,b}$ is the set of $\la$ contained in the Young diagram 
$\g = (b^a)$ of rectangular shape in the above sense.  In particular, 
$\CP_n^{d,d+1}$ coincides with $\CP_n^d$, and $\CP_n^{2d,2d}$ (under the
identification $(\a,\b) = (\b,\a)$) coincides with $\wt \CP_n^d$.
\par
Put 
\begin{equation*}
\tag{7.8.3}
f_{a,b}(w_{\Br}) = 
  \sum_{(\a,\b) \in \CP_n^{a,b}}
     (-1)^{|\a|}\Tr(T_{w_{\Br}}, E_u^{(\a,\b)}).
\end{equation*}
We take $r_k = n, r_{k-1} = n - (2a + 2b -6)$ so that in applying 
Theorem 7.4 or Theorem 7.6, 
$\mu^{(k)}/\mu^{(k-1)}$ is a broken border strip 
of length $2a + 2b -6$.
Let $X$ be a broken border strip of length $2a +2b -6$ contained 
in $\la = (\a,\b) \in \CP^{a,b}_n$.
We can write $X = Y \coprod Z$ with $Y \subset \a$, 
$Z \subset \b$, broken border strips.  
Since the maximum length of a border strip is $a + b -1$,
we have only to consider the following 5 cases.  
\par\medskip
Case I. \ $|Y| = a + b - 1, |Z| = a + b - 5$, 
\par
Case II. \ $|Y| = a + b - 2, |Z| = a + b - 4$,
\par
Case III. \ $|Y| = a + b - 3, |Z| = a + b - 3$, 
\par
Case IV. $|Y| = a + b - 4, |Z| = a + b - 2$,
\par
Case V. $|Y| = a + b - 5, |Z| = a + b- 1$.
\par\medskip
We consider the diagram $\g = (b^a)$ of rectangular 
shape so that $\a \cup \b^* = \g$.  
In the following discussion, we regard $Y$ and $Z$ as paths
in $\g$, instead of considering $\a$ and $\b$ separately.
For example, the following figure explains an example 
of the case I, where $\la = (\a, \b) = (42^21^2, 4^221)$ 
with $n = 20$, 
$Y$ is a  border strip of length 8 and $Z$ is a border strip 
of length 4.  In the figure, $\bullet$ (resp. $\times$) denotes the 
starting point and the ending point of $Y$ (resp. $Z$).
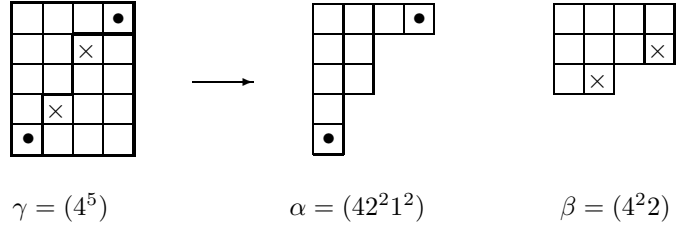
\begin{figure}[h]
\begin{picture}(250,70)
\setlength{\unitlength}{0.8mm}
\thicklines
\put(0,0){\framebox(20, 25)}
\thinlines
\multiput(0,5)(0,5){4}{\line(1,0){20}}
\multiput(5,0)(5,0){3}{\line(0,1){25}}
\thicklines
\put(5,0){\line(0,1){10}}
\put(5,10){\line(1,0){5}}
\put(10,10){\line(0,1){10}}
\put(10,20){\line(1,0){10}}

\put(1.5,1.5){$\bullet$}
\put(16.5, 21.5){$\bullet$}
\put(5.5, 6.0){$\times$}
\put(10.5, 16.0){$\times$}

\put(0,-10){$\g = (4^5)$}

\thinlines
\put(30, 12){\vector(1,0){10}}

\put(50, 25){\line(0,-1){25}}
\put(55, 25){\line(0,-1){25}}
\put(60, 25){\line(0,-1){15}}
\put(65, 25){\line(0,-1){5}}
\put(70, 25){\line(0,-1){5}}

\put(50,0){\line(1,0){5}}
\put(50,5){\line(1,0){5}}
\put(50,10){\line(1,0){10}}
\put(50,15){\line(1,0){10}}
\put(50,20){\line(1,0){20}}
\put(50,25){\line(1,0){20}}

\put(51.5, 1.5){$\bullet$}
\put(66.5, 21.5){$\bullet$}

\put(46,-10){$\a = (42^21^2)$}

\put(90,10){\line(0,1){15}}
\put(95,10){\line(0,1){15}}
\put(100,10){\line(0,1){15}}
\put(105,15){\line(0,1){10}}
\put(110,15){\line(0,1){10}}

\put(90,10){\line(1,0){10}}
\put(90,15){\line(1,0){20}}
\put(90,20){\line(1,0){20}}
\put(90,25){\line(1,0){20}}

\put(95.5, 11){$\times$}
\put(105.5, 16){$\times$}

\put(91, -10){$\b = (4^22)$}
\end{picture}
\end{figure}
\vspace*{1cm}
\par
In computing $f_{a,b}(w_{\Br})$, we use the following 
cancellation property.
\begin{lemma} 
Let $x$ be the top rightmost node of $\g$, 
and $y$ the west of $x$, $z$ the south of $x$.  Assume 
that $x, y \in Y$ and that $z \in Z$. 
Then $\a : a_1 \ge \a_2 \ge \cdots, \b : \b_1\ge \b_2 \ge \cdots$
with $\a_1 = a, \b_1 = b-1$.
Let $\la' = (\a', \b') \in \CP_n$ be defined by
$a_1' = \a_1-1$, $\b_1' = \b_1+1$ and 
$\a'_j = \a_j, \b'_j = \b_j$  for $j \ne 1$.  Then 
$\la' \in \CP_n^{a,b}$.  Put $Y' = Y - \{ x\}$ and 
$Z' = Z + \{ x\}$. 
Then $Y'$ (resp. $Z'$) is a broken border strip of $\a'$
(resp. $\b'$).  Let $X' = Y' \coprod Z'$.  Then we have 
\begin{equation*}
\vD(X) = \vD(X').
\end{equation*}
Moreover, the double partition $(\a - Y, \b - Z)$ coincides with
$(\a' - Y', \b - Z')$.
Hence in the computation of $f_{a,b}(w_{\Br})$,  
the broken border strip $X$ of this type may be ignored.
Similar situations occur also for the cases, such as 
$x$ is the bottom leftmost node of $\g$ and $y$ is the
 north of $x$, $z$ is the east of $x$, and $x,y \in Y$, 
$z \in Z$. 
\end{lemma}
\begin{proof}
Since $y \in Y$ and $z \in Z$, the number of border strips
in $X$ is the same as the numbers in $X'$.  Since 
$c(Y') = c(Y) - 1, c(Z') = c(Z) + 1$, and 
$r(Y') = r(Y), r(Z') = r(Z)$, we see that $\vD(X) = \vD(X')$.
Clearly, we have $(\a - Y, \b - Z) = (\a' - Y', \b' - Z')$.
It follows that in the sequence 
$\emptyset \subset \mu^{(1)} 
   \subset \cdots \subset \mu^{(k)} = \la$ in 
Theorem 7.4 or Theorem 7.6, 
$\emptyset \subset \cdots \subset \mu^{(k-1)}$ is common 
for $\la'$ if $\mu^{(k)}/\mu^{(k-1)} = X$.  Since 
$|\a'| = |\a| -1 $, two terms starting from $X$ and from $X'$
are canceled in the computation of $f_{a,b}(w_{\Br})$ by using
these theorems.
\end{proof}
\para{7.10.}
We shall classify the broken border strip $X = Y \coprod Z$
which is needed for the computation of $f_{a,b}(w_{\Br})$.
Let $x_t$ be the top rightmost node and $x_b$ be 
the bottom leftmost node of the Young diagram $\a$.  
Let $y_t$ be the top rightmost node of $\b$ and $y_b$ be
the bottom leftmost node of $\b$.  Note that we embed 
$\b$ in $\g$ by using $\b^*$, and regard $y_t, y_b$ as 
a box in $\g$. 
\par
First we consider the Case I.
We have $\a_1 = b, \a_a \ne 0$,
and $x_t$ is the top rightmost node and $x_b$ is the
bottom leftmost node of $\g$. $Y$ is a unique
border strip connecting $x_t$ and $x_b$.
Assume that $\a_2 < b$ and $\a_a = 1$.  
Then $y_t$ is the south of $x_t$.
If $Z$ contains $y_t$ or $y_b$, then Lemma 7.9 can be applied, and 
we can ignore this $X$.  Hence $Z$ does not contain $y_t, y_b$.
In this case, $Z$ is a unique border strip connecting the node
west of $y_t$ and the north of $y_b$. 
Next assume that $\a_2 = b$ and $\a_a = 1$. If $y_b \in Z$, 
then Lemma 7.9 can be applied.  So we may assume that $y_b \notin Z$. 
In this case $Z$ is a unique border strip connecting the node north of
$y_b$ and $y_t$ which is the two node south of $x_t$. 
Next assume that $\a_2  < b$ and $\a_a = 2$.  Then $y_t$ is the south 
of $x_t$, and if $y_t \in Z$, the lemma can be applied.  So we may
assume that $y_t \notin Z$,  and $Z$ is a unique border strip 
connecting the node west of $y_t$ and $y_b$.
Finally, assume that $\a_2 = b$ and $\a_a = 2$.  In this case, 
$Z$ is a unique border strip connecting $y_t$ and $y_b$.
\par
Thus Case I is divided into 4 classes, and in each case, $Y$ and $Z$
are determined uniquely by $\la = (\a,\b)$.
In a similar way, one can classify all the possible broken border 
strips $X$ for Case II and Case III.  
Case IV is symmetric to Case II, 
and each class is obtained from the class in Case II, 
by rotating the diagram $\g$ by the angle $180^{\circ}$, 
and then replacing $Y$ and $Z$.
Similarly, Case V is obtained from Case I.
We shall list up all the possible cases for the cases I, II, III, 
in the list below, assuming 
that $a, b \ge 4$.  Case I is divided into 4 classes, 
Case II into 6 classes, and Case III into
12 classes, (Case IV : 6 classes, Case V: 4 classes).       
Here $\bullet$ (resp. $\times$) denotes the starting node and the ending 
node of the border strip $Y$ (resp. $Z$).
In each case, $Y$ and $Z$ are determined uniquely 
by $\la = (\a, \b)$.  Or alternately, if we draw in the diagram $\g$ 
the path connecting nodes marked by $\bullet$, 
so that it is compatible with the path connecting
boxes marked by $\times$, then it determines $\la = (\a, \b)$ uniquely.
\par
Note that in each case, $\vD(X)$ does not depend on $\la$ belonging 
to the class, and has a common value.  We have listed those $\vD(X)$
for each class, where $U = (u^{1/2} - u^{-1/2})$.  
For Case IV or Case V, $\vD(X)$ is obtained from the corresponding 
$\vD(X)$ for Case II or Case I, by replacing $u^{1/2} \lra -u^{-1/2}$.
\par
\vspace{1.5cm}
{\bf Case I.}  \ $|Y| = a+b-1, |Z| = a+b-5$.
\begin{figure}[h]
\begin{picture}(320,70)
\setlength{\unitlength}{0.8mm}


\put(0,0){\framebox(25, 30)}
\put(0,5){\line(1,0){10}}
\put(5,0){\line(0,1){10}}
\put(10,5){\line(0,1){5}}
\put(5,10){\line(1,0){5}}

\put(15, 25){\line(1,0){10}}
\put(15, 20){\line(1,0){5}}
\put(15, 20){\line(0,1){5}}
\put(20, 20){\line(0,1){10}}

\put(1.5, 1.5){$\bullet$}
\put(6.0, 6.0){$\times$}
\put(16.0, 21.0){$\times$}
\put(21.5, 26.5){$\bullet$}

\put(3, -10){(1) \ \ $-U$}

\put(40,0){\framebox(25, 30)}
\put(40,5){\line(1,0){10}}
\put(45,0){\line(0,1){10}}
\put(50,5){\line(0,1){5}}
\put(45,10){\line(1,0){5}}

\put(60, 25){\line(1,0){5}}
\put(60, 20){\line(1,0){5}}
\put(60, 15){\line(1,0){5}}
\put(60, 15){\line(0,1){15}}

\put(41.5, 1.5){$\bullet$}
\put(46.0, 6.0){$\times$}
\put(61.0, 16.0){$\times$}
\put(61.5, 26.5){$\bullet$}

\put(41, -10){(2) \ \ $u\iv U$}


\put(80,0){\framebox(25, 30)}
\put(80,5){\line(1,0){15}}
\put(85,0){\line(0,1){5}}
\put(90,0){\line(0,1){5}}
\put(95,0){\line(0,1){5}}

\put(95, 25){\line(1,0){10}}
\put(95, 20){\line(1,0){5}}
\put(100,20){\line(0,1){10}}
\put(95, 20){\line(0,1){5}}

\put(81.5, 1.5){$\bullet$}
\put(91.0, 1.0){$\times$}
\put(96.0, 21.0){$\times$}
\put(101.5, 26.5){$\bullet$}

\put(83, -10){(3) \ \ $u U$}

\put(120,0){\framebox(25, 30)}
\put(120,5){\line(1,0){15}}
\put(125,0){\line(0,1){5}}
\put(130,0){\line(0,1){5}}
\put(135,0){\line(0,1){5}}

\put(140, 25){\line(1,0){5}}
\put(140, 20){\line(1,0){5}}
\put(140, 15){\line(1,0){5}}
\put(140, 15){\line(0,1){15}}

\put(121.5, 1.5){$\bullet$}
\put(131.0, 1.0){$\times$}
\put(141.5, 16.5){$\times$}
\put(141.5, 26.5){$\bullet$}

\put(123, -10){(4) \ \ $-U$}
\end{picture}
\end{figure}
\vspace*{1.5cm}
\par
{\bf Case II.} \ $|Y| = a+b-2, |Z| = a+b-4$.
\begin{figure}[h]
\begin{picture}(320,70)
\setlength{\unitlength}{0.8mm}


\put(0,0){\framebox(25, 30)}
\put(0,5){\line(1,0){10}}
\put(5,0){\line(0,1){10}}
\put(10,0){\line(0,1){10}}
\put(5,10){\line(1,0){5}}

\put(15, 25){\line(1,0){10}}
\put(15, 20){\line(1,0){5}}
\put(15, 20){\line(0,1){5}}
\put(20, 20){\line(0,1){10}}

\put(1.5, 1.5){$\bullet$}
\put(16.0, 21.0){$\times$}
\put(21.5, 26.5){$\bullet$}
\put(6.0, 1.0){$\times$}
\put(6.5, 6.5){$\bullet$}

\put(0, -10){(1) \ \ $u^{1/2}U^2$}

\put(40,0){\framebox(25, 30)}
\put(40,5){\line(1,0){10}}
\put(45,0){\line(0,1){10}}
\put(50,5){\line(0,1){5}}
\put(45,10){\line(1,0){5}}

\put(55, 25){\line(1,0){10}}
\put(55, 20){\line(1,0){10}}
\put(55, 20){\line(0,1){5}}
\put(60, 20){\line(0,1){10}}

\put(41.5, 1.5){$\bullet$}
\put(46.0, 6.0){$\times$}
\put(61.0, 21.0){$\times$}
\put(56.5, 21.5){$\bullet$}
\put(61.5, 26.5){$\bullet$}

\put(38, -10){(2) \ \ $-u^{1/2}U^2$}


\put(80,0){\framebox(25, 30)}
\put(80,5){\line(1,0){10}}
\put(85,0){\line(0,1){10}}
\put(90,5){\line(0,1){5}}
\put(80,10){\line(1,0){10}}

\put(95, 25){\line(1,0){10}}
\put(95, 20){\line(1,0){5}}
\put(100,20){\line(0,1){10}}
\put(95, 20){\line(0,1){5}}

\put(81.0, 1.0){$\times$}
\put(86.0, 6.0){$\times$}
\put(81.5, 6.5){$\bullet$}
\put(96.0, 21.0){$\times$}
\put(101.5, 26.5){$\bullet$}

\put(80, -10){(3) \ \ $u^{1/2}U^2$}

\put(120,0){\framebox(25, 30)}
\put(120,5){\line(1,0){10}}
\put(125,0){\line(0,1){10}}
\put(130,5){\line(0,1){5}}
\put(120,10){\line(1,0){10}}

\put(140, 25){\line(1,0){5}}
\put(140, 20){\line(1,0){5}}
\put(140, 15){\line(1,0){5}}
\put(140, 15){\line(0,1){15}}

\put(121.0, 1.0){$\times$}
\put(126.0, 6.0){$\times$}
\put(121.5, 6.5){$\bullet$}

\put(141.5, 16.5){$\times$}
\put(141.5, 26.5){$\bullet$}

\put(119, -10){(4) \ \ $-u^{-1/2}U^2$}

\end{picture}
\end{figure}
\par
\vspace*{0.8cm}
\begin{figure}[h]
\begin{picture}(320,80)
\setlength{\unitlength}{0.8mm}


\put(0,0){\framebox(25, 30)}
\put(0,5){\line(1,0){10}}
\put(5,0){\line(0,1){10}}
\put(10,5){\line(0,1){5}}
\put(5,10){\line(1,0){5}}

\put(15, 25){\line(1,0){10}}
\put(15, 20){\line(1,0){5}}
\put(15, 20){\line(0,1){10}}
\put(20, 20){\line(0,1){10}}

\put(1.5, 1.5){$\bullet$}
\put(16.0, 21.0){$\times$}
\put(16.5, 26.5){$\bullet$}
\put(21.0, 26.0){$\times$}
\put(6.0, 6.0){$\times$}

\put(0, -10){(5) \ \ $-u^{-1/2}U^2$}

\put(40,0){\framebox(25, 30)}
\put(40,5){\line(1,0){15}}
\put(45,0){\line(0,1){5}}
\put(50,0){\line(0,1){5}}
\put(55,0){\line(0,1){5}}

\put(55, 25){\line(1,0){10}}
\put(55, 20){\line(1,0){5}}
\put(55, 20){\line(0,1){10}}
\put(60, 20){\line(0,1){10}}

\put(41.5, 1.5){$\bullet$}
\put(51.0, 1.0){$\times$}
\put(56.5, 26.5){$\bullet$}
\put(56.0, 21.0){$\times$}
\put(61.0, 26.0){$\times$}

\put(38, -10){(6) \ \ $u^{1/2}U^2$}

\end{picture}
\end{figure}
\par
\vspace*{1.5cm}
{\bf Case III.} \ $|Y| = a+b-3, |Z| = a+b-3$.
\par
\begin{figure}[h]
\begin{picture}(320,70)
\setlength{\unitlength}{0.8mm}


\put(0,0){\framebox(25, 30)}
\put(0,5){\line(1,0){10}}
\put(5,0){\line(0,1){10}}
\put(10,0){\line(0,1){10}}
\put(5,10){\line(1,0){5}}

\put(15, 25){\line(1,0){10}}
\put(15, 20){\line(1,0){10}}
\put(15, 20){\line(0,1){5}}
\put(20, 20){\line(0,1){10}}

\put(1.5, 1.5){$\bullet$}
\put(6.0, 1.0){$\times$}
\put(6.5, 6.5){$\bullet$}
\put(21.5, 21.5){$\times$}
\put(21.5, 26.5){$\bullet$}
\put(16.5, 21.5){$\bullet$}

\put(3, -10){(1) \ \ $U^3$}

\put(40,0){\framebox(25, 30)}
\put(40,5){\line(1,0){10}}
\put(45,0){\line(0,1){10}}
\put(50,5){\line(0,1){5}}
\put(45,10){\line(1,0){5}}

\put(55, 25){\line(1,0){10}}
\put(55, 20){\line(1,0){5}}
\put(55, 20){\line(0,1){5}}
\put(60, 20){\line(0,1){10}}

\put(41.5, 1.5){$\bullet$}
\put(46.0, 6.0){$\times$}
\put(61.0, 26.0){$\times$}
\put(56.5, 21.5){$\bullet$}

\put(41, -10){(2) \ \ $-U$}


\put(80,0){\framebox(25, 30)}
\put(80,5){\line(1,0){10}}
\put(85,0){\line(0,1){10}}
\put(90,0){\line(0,1){10}}
\put(85,10){\line(1,0){5}}

\put(95, 25){\line(1,0){10}}
\put(95, 20){\line(1,0){5}}
\put(100,20){\line(0,1){10}}
\put(95, 20){\line(0,1){10}}

\put(81.5, 1.5){$\bullet$}
\put(86.5, 6.5){$\bullet$}
\put(86.0, 1.0){$\times$}
\put(96.0, 21.0){$\times$}
\put(101.0, 26.0){$\times$}
\put(96.5, 26.5){$\bullet$}

\put(83, -10){(3) \ \ $U^3$}


\put(120,0){\framebox(25, 30)}
\put(120,5){\line(1,0){15}}
\put(125,0){\line(0,1){5}}
\put(130,0){\line(0,1){5}}
\put(135,0){\line(0,1){5}}

\put(135, 25){\line(1,0){10}}
\put(135, 20){\line(1,0){5}}
\put(135, 20){\line(0,1){5}}
\put(140, 20){\line(0,1){10}}

\put(121.5, 1.5){$\bullet$}
\put(131.0, 1.0){$\times$}
\put(136.5, 21.5){$\bullet$}
\put(141.0, 26.0){$\times$}

\put(123, -10){(4) \ \ $uU$}

\end{picture}
\end{figure}
\par
\vspace*{1.5cm}
\begin{figure}[h]
\begin{picture}(320,70)
\setlength{\unitlength}{0.8mm}


\put(0,0){\framebox(25, 30)}
\put(0,5){\line(1,0){10}}
\put(5,0){\line(0,1){10}}
\put(10,5){\line(0,1){5}}
\put(0,10){\line(1,0){10}}

\put(15, 25){\line(1,0){10}}
\put(15, 20){\line(1,0){5}}
\put(15, 20){\line(0,1){10}}
\put(20, 20){\line(0,1){10}}

\put(1.0, 1.0){$\times$}
\put(1.5, 6.5){$\bullet$}
\put(6.0, 6.0){$\times$}
\put(16.5, 26.5){$\bullet$}
\put(21.0, 26.0){$\times$}
\put(16.0, 21.0){$\times$}

\put(3, -10){(5) \ \ $U^3$}

\put(40,0){\framebox(25, 30)}
\put(40,5){\line(1,0){10}}
\put(45,0){\line(0,1){10}}
\put(50,5){\line(0,1){5}}
\put(45,10){\line(1,0){5}}

\put(50, 25){\line(1,0){15}}
\put(50, 25){\line(0,1){5}}
\put(55, 25){\line(0,1){5}}
\put(60, 25){\line(0,1){5}}

\put(41.5, 1.5){$\bullet$}
\put(46.0, 6.0){$\times$}
\put(61.0, 26.0){$\times$}
\put(51.5, 26.5){$\bullet$}

\put(41, -10){(6) \ \ $u\iv U$}


\put(80,0){\framebox(25, 30)}
\put(80,5){\line(1,0){15}}
\put(85,0){\line(0,1){5}}
\put(90,0){\line(0,1){5}}
\put(95,0){\line(0,1){5}}

\put(90, 25){\line(1,0){15}}
\put(90, 25){\line(0,1){5}}
\put(95, 25){\line(0,1){5}}
\put(100, 25){\line(0,1){5}}

\put(81.5, 1.5){$\bullet$}
\put(91.0, 1.0){$\times$}

\put(101.0, 26.0){$\times$}
\put(91.5, 26.5){$\bullet$}

\put(83, -10){(7) \ \ $-U$}


\put(120,0){\framebox(25, 30)}
\put(120,5){\line(1,0){10}}
\put(120,10){\line(1,0){10}}
\put(125,0){\line(0,1){10}}
\put(130,5){\line(0,1){5}}

\put(135, 25){\line(1,0){10}}
\put(135, 20){\line(1,0){10}}
\put(135, 20){\line(0,1){5}}
\put(140, 20){\line(0,1){10}}

\put(121.0, 1.0){$\times$}
\put(126.0, 6.0){$\times$}
\put(121.5, 6.5){$\bullet$}

\put(136.5, 21.5){$\bullet$}
\put(141.5, 26.5){$\bullet$}
\put(141.0, 21.0){$\times$}

\put(123, -10){(8) \ \ $U^3$}

\end{picture}
\end{figure}
\par
\newpage
\vspace*{0.5cm}
\begin{figure} [h]
\begin{picture}(320,60)
\setlength{\unitlength}{0.8mm}

\put(0,0){\framebox(25, 30)}
\put(0,5){\line(1,0){10}}
\put(5,0){\line(0,1){10}}
\put(10,5){\line(0,1){5}}
\put(5,10){\line(1,0){5}}

\put(15, 25){\line(1,0){10}}
\put(15, 20){\line(1,0){5}}
\put(15, 20){\line(0,1){5}}
\put(20, 20){\line(0,1){10}}

\put(1.0, 1.0){$\times$}
\put(6.5, 6.5){$\bullet$}

\put(21.5, 26.5){$\bullet$}
\put(16.0, 21.0){$\times$}

\put(3, -10){(9) \ \ $-U$}

\put(40,0){\framebox(25, 30)}
\put(40,5){\line(1,0){10}}
\put(45,0){\line(0,1){10}}
\put(50,5){\line(0,1){5}}
\put(45,10){\line(1,0){5}}

\put(60, 25){\line(1,0){5}}
\put(60, 20){\line(1,0){5}}
\put(60, 15){\line(1,0){5}}
\put(60, 15){\line(0,1){15}}

\put(41.0, 1.0){$\times$}
\put(46.5, 6.5){$\bullet$}

\put(61.5, 26.5){$\bullet$}
\put(61.0, 16.0){$\times$}

\put(41, -10){(10) \ \ $u\iv U$}


\put(80,0){\framebox(25, 30)}
\put(80,5){\line(1,0){5}}
\put(80,10){\line(1,0){5}}
\put(80,15){\line(1,0){5}}
\put(85,0){\line(0,1){15}}

\put(95, 25){\line(1,0){10}}
\put(95, 20){\line(1,0){5}}
\put(95, 20){\line(0,1){5}}
\put(100, 20){\line(0,1){10}}

\put(81.0, 1.0){$\times$}
\put(81.5, 11.5){$\bullet$}

\put(101.5, 26.5){$\bullet$}
\put(96.0, 21.0){$\times$}

\put(83, -10){(11) \ \ $uU$}


\put(120,0){\framebox(25, 30)}
\put(120,5){\line(1,0){5}}
\put(120,10){\line(1,0){5}}
\put(120,15){\line(1,0){5}}
\put(125,0){\line(0,1){15}}

\put(140, 25){\line(1,0){5}}
\put(140, 20){\line(1,0){5}}
\put(140, 15){\line(1,0){5}}
\put(140, 15){\line(0,1){15}}

\put(121.0, 1.0){$\times$}
\put(121.5, 11.5){$\bullet$}

\put(141.5, 26.5){$\bullet$}
\put(141.0, 16.0){$\times$}

\put(123, -10){(12) \ \ $-U$}
\end{picture}
\end{figure}
\vspace*{1.5cm}

\para{7.11.}
In each case listed in 7.10, one obtains a unique $\la' \in \CP_{n'}$
($n' = n- (2a + 2b -6)$) from $\la$ by removing the broken border strip
$X = Y\coprod Z$, where $\la' = (\a', \b')$ with 
$\a' = \a - Y, \b' = \b - Z$.
Let $J$ be the set of classes in the list in 7.10, and let 
$\CQ_j$ be the set of $\la' \in \CP_{n'}$ satisfying the condition 
in each case $j$. 
For example we consider the case $j = $ I-(1).  Then 
$\CQ = \CQ_{I-(1)}$ is the set of $\la' = (\a', \b')$ satisfying the 
following conditions.
\begin{equation*}
\begin{cases}
\a' = (\a'_1, \a'_2, \dots, \a'_{a-3}), \\
{\b'}^* = ({\b'}^*_1, {\b'}_2^*, \dots, {\b'}^*_{a-1}) \quad\text{ with }\quad
              {\b'}^*_1 = b-1, {\b'}^*_{a-1} = 1, \\
\a_i + {\b'}^*_{a-1-i} = b-2 \text{ for } i = 1, \dots, a-3. 
\end{cases}
\end{equation*}
For each $\la' = (\a', \b') \in \CQ$, 
we consider a broken border strip $X'$ 
of length $2a + 2b -10$.  Then $X'$ is unique, and is given as 
$X' = Y'\coprod Z'$, where $Y'$ is a unique border strip in $\a'$
of length $a + b -7$, and $Z'$ is a unique border strip in $\b'$
of length $a + b -3$.  By removing $X'$ from $\la'$, one obtains 
$\la'' \in \CP_{n''}$, where $n'' = n' - (2a + 2b -10) = (a-4)(b-4)$. 
Then it is easy to check that
\begin{equation*}
\tag{7.11.1}
\{ \la'' \in \CP_{n''} \mid \la' \in \CQ \} = \CP_{n''}^{a-4, b-4},
\end{equation*}
and one recovers the original set $\CP_{n}^{a-4, b-4}$ by replacing 
$a,b$ by $a-4, b-4$.
We can compute that $\vD(X') = -(u^{1/2} - u^{-1/2})$, which is 
independent from $\la' \in \CQ$.
For example, let  $a = 7, b = 6$ with $n = 42$.  Take 
$\la = (6432^21^2, 6^2542) \in \CP_{42}^{7,6}$.
Then $\la' = (\a', \b') = (321^2, 6431^2)$, and under an appropriate
rearrangement, $\g' = \a' \cup {\b'}^*$ can be drawn as in the 
following figure. Here $Y'$ (resp. $Z'$) is a unique border strip of 
length 6, (resp. length 10), and $\bullet$ (resp. $\times$ ) denotes
the starting point and the ending point of $Y'$ (resp. $Z'$).  
From this, we obtain $\g'' = (2^3)$, and 
$\la'' = (\a'', \b'') = (1, 32) \in \CP_6^{3,2}$.

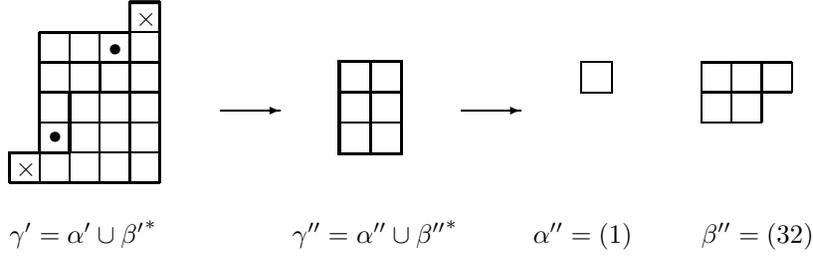
\begin{figure}[h]
\begin{picture}(300,70)
\setlength{\unitlength}{0.8mm}

\thicklines
\put(0,0){\line(1,0){25}}
\put(0,0){\line(0,1){5}}
\put(0,5){\line(1,0){5}}
\put(5,5){\line(0,1){20}}
\put(5,25){\line(1,0){15}}
\put(20,25){\line(0,1){5}}
\put(20,30){\line(1,0){5}}
\put(25,30){\line(0,-1){30}}

\thinlines
\multiput(5,5)(0,5){5}{\line(1,0){20}}
\multiput(5,0)(5,0){4}{\line(0,1){25}}

\thicklines
\put(5,5){\line(1,0){5}}
\put(10,5){\line(0,1){10}}
\put(10,15){\line(1,0){5}}
\put(15,15){\line(0,1){5}}
\put(15,20){\line(1,0){5}}
\put(20,20){\line(0,1){5}}

\put(1.0,1.0){$\times$}
\put(6.5, 6.5){$\bullet$}
\put(16.5, 21.0){$\bullet$}
\put(21.0, 26.0){$\times$}

\put(0, -10){$\g' = \a' \cup {\b'}^*$}

\thinlines
\put(35, 12){\vector(1,0){10}}

\thicklines
\put(55, 5){\framebox(10,15)}
\thinlines
\put(60,5){\line(0,1){15}}
\put(55,10){\line(1,0){10}}
\put(55,15){\line(1,0){10}}

\thicklines
\put(55,15){\line(1,0){5}}
\put(60,15){\line(0,1){5}}

\put(47, -10){$\g'' = \a'' \cup {\b''}^*$}

\thinlines
\put(75, 12){\vector(1,0){10}}

\thinlines
\put(95, 15){\framebox(5,5)}

\put(115, 10){\line(1,0){10}}
\put(115, 15){\line(1,0){15}}
\put(115, 20){\line(1,0){15}}
\put(115, 10){\line(0,1){10}}
\put(120, 10){\line(0,1){10}}
\put(125, 10){\line(0,1){10}}
\put(130, 15){\line(0,1){5}}

\put(87,-10){$\a'' = (1)$}
\put(115, -10){$\b'' = (32)$}
\end{picture}
\end{figure}

\par
\vspace*{1.0cm}
In fact, a similar kind of arguments work for all other cases, 
and the set $Q_j$ is described in a similar way.  In particular, 
$\g' = \a' \cup {\b'}^*$ is of the shape obtained from a rectangle 
by attaching two nodes, one on the above or right of the northeast 
corner, and the other on the below or left of the southwest corner of 
the rectangle.  
In all the cases, the broken border strip $X' = Y' \coprod Z'$ of 
length $2a + 2b -10$ is determined uniquely, 
and we always find
the set $\CP_{n''}^{a-4, b-4}$ after removing the border strips $X'$.
Moreover, $\vD(X')$ takes the common value $ -(u^{1/2} - u^{-1/2})$ 
for all the cases through Case I $\sim$ Case V.  
\par
Recall that $w_{\Br} = (r_1, \dots, r_k)$ with $r_k = n$.
Assume that 
\begin{equation*}
r_k - r_{k-1} = 2a + 2b - 6, \quad r_{k-1}-r_{k-2} = 2a + 2b - 10.
\end{equation*}
Thus $r_{k-2} = n''$ with $n'' = (a - 4)(b - 4)$.
We put $\Br'' = (r_1, \dots, r_{k-2})$ and consider
$w_{\Br''} \in W_{n''}$.
By investigating the above list, we have the following lemma.
\begin{lemma} 
Under the notation above, there exists a non-zero (Laurent) 
polynomial $h(u)$ such that 
\begin{equation*}
f_{a,b}(w_{\Br}) = h(u)f_{a-4,b-4}(w_{\Br''}).
\end{equation*}
\end{lemma}
\begin{proof}
For each $j \in J$, we denote by $\vD_j(u)$ the Laurent polynomial
$\vD(X)$ attached to the broken border strip $X$ for $Q_j$ as in the 
list.  We put $\ve_j = (-1)^{|\a|}$ for 
$\la = (\a, \b) \in Q_j$.  (Note that $\ve_j$ is independent of 
the choice of $\la \in \CQ_j$). By (7.11.1), the cardinality of
$Q_j$ coincides with the cardinality of $\CP_{n''}^{a-4,b-4}$, 
hence is independent of $j \in J$.  Then the investigation in 7.11 
shows that 
\begin{equation*}
\tag{7.12.1}
f_{a,b}(w_{\Br}) = \biggl\{(u^{1/2} - u^{-1/2})|\CP_{n''}^{a-4,b-4}|
                     \sum_{j \in J}
                       \ve_j\D_j(u)\biggr\}f_{a-4,b-4}(w_{\Br''}).
\end{equation*}
Thus in order to show the lemma, it is enough to see that 
$\sum_j\ve_j\vD_j(u) \ne 0$.
For this we compare the highest degree term $u^{3/2}$ in $\vD_j$.
It follows from the list in 7.10, $\vD_j$ contains the term $u^{2/3}$
in the following cases, where the coefficients are always 1.
\par\medskip
Case I. \quad (3),
\par
Case II. \ \,(1), (3), (6),
\par
Case III. \  (1), (3), (4), (5), (8), (11),
\par
Case IV. \ (2), (4), (5),
\par
Case V. \hspace{2.5mm}(2),
\par\medskip\noindent
where the numbering in Case IV and V is given by 
the bijective correspondence with Case II and Case I
through the rotation of $\g$.
Note that $\ve_j$ takes the constant value for each case, I $\sim$ V. 
They have the common value for Case I, III, or V, and have a different
common value for Case II or IV.  This shows that the coefficient of $u^{2/3}$
in $\sum_j\ve_j\vD_j \ne 0$.  Hence $h(u) \ne 0$ as asserted.
\end{proof}
\par
Returning to the original setting, we show the following 
two propositions, 
which give the proof of Proposition 6.6.
\begin{prop}  
Assume that $W$ is of type $B_n$ with $n = d^2 + d$. 
We define an element $w_{\Br} \in W$ 
by the cycle type expression given in (7.3.2) as follows.
\begin{equation*}
\begin{aligned}
&[\ol 2,12,16,12+16, 16+16, \dots, 12 + 16k, 16+16k],
    &\quad\text{ if } d \equiv 1 \pmod 4,  \\ 
&[6, 16, 20, 16+16, 20+16,\dots, 16+16k, 20+16k],  
    &\quad\text{ if } d \equiv 2 \pmod 4, \\
&[4,8, 4+16, 8+16, \dots, 4+16k, 8+16k],
    &\quad\text{ if } d \equiv 3 \pmod 4, \\ 
&[8,12,8+16, 12+16, \dots, 8+16k, 12+16k],
    &\quad\text{ if } d \equiv 0 \pmod 4, \\
\end{aligned}
\end{equation*}
for some $ k \ge 1$, 
where the last term is equal to $4d-4$, and the next  
term is equal to $4d-8$, and so on. 
Then we have $f_{\vL_c}(w_{\Br}) \ne 0$.
\end{prop}
\begin{proof}
We apply Lemma 7.12 with $a = d+1, b = d$.
Then $f_{\vL_c}(w_{\Br}) = h(u)f_{\vL_c}(w_{\Br''})$
for some non-zero $h(u)$, where 
$\Br = (r_1, \dots, r_k)$ and 
$\Br'' = (r_1, \dots, r_{k-2})$ with 
$r_k = n = d(d+1)$, $r_k  - r_{k-1} = 4d-4$, 
$r_{k-1} - r_{k-2} = 4d-8$, $r_{k-2} = (d-4)(d-3)$.
Thus the computation of $f_{\vL_c}(w_{\Br})$ is reduced to
the case where $d = 1,2,3$.  Assume that $d = 1$, then 
$n = 2$.  One can check by using the formula 
for $\ol\vD(X)$ that $f_{\vL_c}(w) \ne 0$ for $w = (\ol 2)$.
(Note that $f_{\vL_c}(w) = 0$ for $w = (2)$). 
Next assume that $d = 2$, then $n = 6$. 
The direct computation shows that  
$f_{\vL_c(w)} \ne 0$ for $w = (6)$.  
Finally assume that $d = 3$, then $n = 12$.
By using a similar method as in the proof of Lemma 7.12, 
one can show that $f_{\vL_c}(w) \ne 0$ for $w = (4,8)$
(we obtain a similar list, but some classes in the list in 
7.10 don't appear for this case).  This proves the pros position.
\end{proof}
\begin{prop}
Assume that $W$ is of type $D_n$ with $n = 4d^2$.  
We define an element $w_{\Br} \in W$ as follows.
\begin{equation*}
\begin{aligned}
&[\ol 1, \ol 3, 14, 18, 14+16, 18+16, \dots, 14 + 16k, 18+16k],
    &\quad\text{ if } d \equiv 1 \pmod 2,  \\ 
&[6, 10, 6+16, 10+16, \dots, 6+16k, 10+16k],  
    &\quad\text{ if } d \equiv 0 \pmod 2, \\
\end{aligned}
\end{equation*}
for some $k \ge 1$, where the last term is equal to 
$4d - 6$ and the next term is equal to $4d - 10$, and so
on.  Then we have $f_{\vL_c}(w_{\Br}) \ne 0$.
\end{prop}
\begin{proof}
We apply Lemma 7.12 with $a = b = 2d$.  
Note that since $\a \ne \b$ for any 
$\la = (\a,\b) \in \CP^{2d,2d}_n$, the argument for type $B_n$
can be applied without change.
Then 
$f_{\vL_c}(w_{\Br}) = h(u)f_{\vL_c}(w_{\Br''})$
for some non-zero $h(u)$, where 
$\Br = (r_1, \dots, r_k)$ and $\Br'' = (r_1, \dots, r_{k-2})$
with $r_k = n = 4d^2$, $r_k-r_{k-1} = 4d-6$, 
$r_{k-1} - r_{k-2} = 4d-10$, $r_{k-2} = 4(d-2)^2$.
Thus the computation of $f_{\vL_c}(w_{\Br})$ is reduced to the
case where $d = 1,2$. Assume that $d = 2$.  Then $n = 16$. 
Since $a = b = 4$, Lemma 7.12 can be applied, and we see 
that $f_{\vL_c}(w) \ne 0$ for $w = (6,10)$.  Next assume that 
$d = 1$.  Then $n = 4$.  One can check by using $\ol\vD'(X)$ 
that $f_{\vL_c}(w) \ne 0$ for $w = (\ol 1, \ol 3)$.
This proves the proposition.
\end{proof}

\begin{remark}  
As the formula (7.12.1) shows, our element 
$f_{\vL_c}(w_{\Br})$ turns out to be 0 if $u \mapsto 1$.   
So our computation cannot be performed in the level of 
Weyl groups.  On the other hand, 
Lusztig showed in [L5] that there exists an element 
$w \in W$ such that $f_{\vL_c}(w)|_{u=1} \ne 0$.  Thus
$f_{\vL_c}(w) \ne 0$ as a polynomial.  This $w$ has  
a simpler form than ours, but since it is a product of 
negative cycles of the full length (i.e., the sum of the
lengths of negative cycles is equal to $n$),  it is not
appropriate for the computation of the bitrace 
on the flags (cf. Lemma 6.10).
\end{remark}


\end{document}